\documentclass[11pt]{amsart}

\usepackage{amsmath,amsthm,amsfonts,amssymb, amscd}
\usepackage[margin=1in]{geometry}
\usepackage[all,cmtip]{xy}
\usepackage{tikz}
\usepackage{hyperref}
\usepackage{verbatim}
\usepackage{todonotes}

\newenvironment{enumalph}{\begin{enumerate}  }{\end{enumerate}}

\newcommand{\set}[1]{\left\{ #1 \right\}}
\newcommand{\abs}[1]{\left| #1 \right|}

\newcommand{\wt}[1]{\widetilde{ #1}}

\newcommand{\ol}[1]{\overline{#1}}

\DeclareMathOperator{\Dist}{Dist}

\DeclareMathOperator{\Ext}{Ext}

\DeclareMathOperator{\gr}{gr}
\newcommand{\grbar}{\ol{\gr}\ }

\DeclareMathOperator{\opH}{H}
\DeclareMathOperator{\Hom}{Hom}

\DeclareMathOperator{\im}{Im}

\DeclareMathOperator{\Lie}{Lie}
\newcommand{\modG}{\mathrm{mod}(G)}
\DeclareMathOperator{\rad}{rad}

\DeclareMathOperator{\res}{res}
\DeclareMathOperator{\soc}{soc}

\DeclareMathAlphabet{\mathpzc}{OT1}{pzc}{m}{it}

\newcommand{\E}{\mathbb{E}}
\newcommand{\F}{\mathbb{F}}
\newcommand{\N}{\mathbb{N}}

\newcommand{\Z}{\mathbb{Z}}

\newcommand{\calR}{\mathcal{R}}

\newcommand{\fraku}{\mathfrak{u}}
\newcommand{\fru}{\mathfrak{u}}

\newcommand{\Fp}{\F_p}
\newcommand{\Fq}{\F_q}

\newcommand{\ofp}{\otimes_{\Fp}}

\newcommand{\ufp}{\fru_{\Fp}}
\newcommand{\ufq}{\fru_{\Fq}}

\newcommand{\Rufq}{\calR_{\Fq/\Fp}(\ufq)}

\newcommand{\Gfp}{G(\Fp)}
\newcommand{\Gfq}{G(\Fq)}

\newcommand{\Bfq}{B(\Fq)}

\newcommand{\Ufq}{U(\Fq)}
\newcommand{\Tfp}{T(\Fp)}
\newcommand{\Tfq}{T(\Fq)}

\numberwithin{equation}{subsection}

\newtheorem{theorem}{Theorem}[subsection]
\newtheorem{proposition}[theorem]{Proposition}

\newtheorem{corollary}[theorem]{Corollary}

\newtheorem{lemma}[theorem]{Lemma}

\newtheorem*{theorem*}{Theorem}
\newtheorem*{lemma*}{Lemma}

\theoremstyle{definition}

\newtheorem{definition}[theorem]{Definition}
\newtheorem{remark}[theorem]{Remark}

\title[First Cohomology Groups with Small Dominant Weights]{First Cohomology for Finite Groups of Lie Type: \\ Simple Modules with Small Dominant Weights}

\thanks{The members of the University of Georgia VIGRE Algebra Group are Brian D.\ Boe, Adrian M.\ Brunyate, Jon F.\ Carlson, Leonard Chastkofsky, Christopher M.\ Drupieski, Niles Johnson, Benjamin F.\ Jones, Wenjing Li, Daniel K.\ Nakano, Nham Vo Ngo, Duc Duy Nguyen, Brandon L.\ Samples, Andrew J.\ Talian, Lisa Townsley, and Benjamin J.\ Wyser.
}

\date{\today}

\author{University of Georgia VIGRE Algebra Group}

\address{
Department of Mathematics \\
University of Georgia \\
Athens, GA~30602-7403}
\address{
Benjamin F. Jones\\
Department of Mathematics, Statistics, and Computer Science\\
University of Wisconsin--Stout\\
Menomonie, Wisconsin 54751}


\begin{document}

\begin{abstract} Let $k$ be an algebraically closed field of characteristic $p > 0$, and let $G$ be a simple, simply connected algebraic group defined over $\mathbb{F}_p$. Given $r \geq 1$, set $q=p^r$, and let $G(\mathbb{F}_q)$ be the corresponding finite Chevalley group. In this paper we investigate the structure of the first cohomology group $\operatorname{H}^1(G(\mathbb{F}_q),L(\lambda))$ where $L(\lambda)$ is the simple $G$-module of highest weight $\lambda$. Under certain very mild conditions on $p$ and $q$, we are able to completely describe the first cohomology group when $\lambda$ is less than or equal to a fundamental dominant weight. In particular, in the cases we consider, we show that the first cohomology group has dimension at most one. Our calculations significantly extend, and provide new proofs for, earlier results of Cline, Parshall, Scott, and Jones, who considered the special case when $\lambda$ is a minimal nonzero dominant weight. 
\end{abstract}

\maketitle

\section{Introduction}

\subsection{}

Let $k$ be an algebraically closed field of characteristic $p>0$, and let $G$ be a simple, simply-connected algebraic group scheme defined over $\F_p$. Let $F: G \rightarrow G$ be the standard Frobenius map on $G$, and let $F^r$ be the $r$-th iterate of $F$. Set $q = p^r$. The $r$-th Frobenius kernel $G_r$ of $G$ is the scheme-theoretic kernel of $F^r$, and the finite Chevalley group $G(\F_q)$ consists of the fixed points in $G$ under $F^r$. It is well-known that the representation theories of $G$, $G_r$ and $G(\F_q)$ are interrelated (cf.\ \cite{Humphreys:2006,Nakano:2010}), and one can relate the cohomology theories via various spectral sequences and limiting techniques \cite{Cline:1977,Bendel:2001}. Even with our current knowledge of these connections, our understanding of the dimensions of first cohomology groups for finite groups with non-trivial coefficients is limited.

In 1984, Guralnick \cite{Guralnick:1986} stated a conjecture for a universal upper bound on the dimension of first cohomology groups for finite groups with coefficients in a faithful simple module. Since that time counterexamples have been found to the strong form of Guralnick's conjecture. For example, Scott \cite{Scott:2003} and others have shown that there exist simple modules for finite Chevalley groups for which the first cohomology group has dimension at least 3. Guralnick in work with Aschbacher 
\cite{Aschbacher:1984} and Hoffman \cite{Guralnick:1998} provided an upper bound on the the dimension of the aforementioned first cohomology group by one-half times the dimension of the given module. More recently, for 
semisimple algebraic groups, Cline, Parshall and Scott have demonstrated an upper bound (depending only on the associated root system of the group) for the first cohomology with coefficients in a simple module. Further work along these lines 
is provided in \cite{PS:2009b}. Results in the cross characteristic case for the dimension of the first cohomology group were proved by Guralnick and Tiep \cite{Guralnick:2009}.

In this paper we investigate the structure of the cohomology group $\opH^1(G(\F_q),L(\lambda))$ where $L(\lambda)$ is a simple $G(\F_q)$-module. In order to understand how the corresponding algebraic group and Lie algebra cohomology are related to this computation we use the powerful filtration techniques developed by Lin and Nakano \cite{Lin:1999} and extended more recently by Friedlander \cite{Friedlander:2010} using the idea of Weil restriction. This approach enables us to employ knowledge about the geometry of the flag variety $G/B$, in particular, Andersen's famous results \cite{Andersen:1984a} on the socle of the sheaf cohomology group ${\mathcal H}^{1}(G/B,{\mathcal L}(\mu))$ for $\mu$ an arbitrary weight.  

With our machinery we are able to give a complete description of $\opH^1(G(\F_q),L(\lambda))$ when $p>2$ for $\Phi=A_n$, $D_n$, $p>3$ for $\Phi=B_n$, $C_n$, $E_6$, $E_7$, $F_4$, $G_2$, $p>5$ for  $\Phi = E_8$, $q>3$,  and $\lambda$ is a fundamental dominant weight. Under certain additional mild restrictions we can improve these results to calculate this cohomology group when $\lambda$ is less than or equal to a fundamental dominant weight. In particular, we show under these restrictions that if $\lambda$ is less than or equal to a fundamental dominant weight, then $\dim \opH^1(G(\F_q),L(\lambda)) \leq 1$.\footnote{Our techniques can also be employed to make calculations for larger dominant weights, provided that one either possesses more detailed information about the structure of the cohomology group $\opH^{1}(U_{1},L(\lambda))$, or that one imposes stronger restrictions on $p$ and $q$.} One can view our results in the framework of Guralnick's conjecture on the size of the first cohomology group when the coefficient module is taken in a certain subcollection of simple modules.

Our work extends the seminal results of Cline, Parshall and Scott \cite{Cline:1975}, of Jones \cite{Jones:1975}, and of Jones and Parshall \cite{JonesParshall:76}, where they considered the special case that $\lambda$ is a minimal nonzero dominant weight. The computations in \cite{Cline:1975} were used in Wiles' proof of Fermat's Last Theorem \cite{Wiles:1995} to show that certain deformation spaces of modular forms and elliptic curves have the same dimension. Our work uses completely different methods than those in \cite{Cline:1975} and may have connections and uses for other number theoretic questions.

\subsection{Main results}

The main results of the paper are stated below. Let $k$ be an algebraically closed field of characteristic $p > 0$, and let $G$ be a simple, simply-connected algebraic group over $k$, which is defined over the field $\F_p$ and with associated root system $\Phi$. Let $r \geq 1$ and set $q = p^r$. For a complete explanation of notation, see Section \ref{subsection:notation}.

Cline, Parshall, Scott and van der Kallen \cite[Theorem 7.4]{Cline:1977} proved that the restriction map $\opH^1(G,L(\lambda)) \rightarrow \opH^1(\Gfq,L(\lambda))$ is injective when $\lambda$ is a $p^{r}$-restricted weight (i.e., $\lambda\in X_{r}(T)$). Our first main result, proved in Section \ref{subsection:isowithG}, is that the restriction map is an isomorphism when $\lambda$ is less than or equal to a fundamental dominant weight, and when the prime $p$ satisfies some fairly mild restrictions depending on the root system.

\begin{theorem} \label{maintheorem1}
Assume that $p>2$ when $\Phi=A_n$, $D_n$, $p>3$ when $\Phi=B_n$, $C_n$, $E_6$, $E_7$, $F_4$, $G_2$, and $p>5$ when $\Phi = E_8$. Suppose $\lambda \leq \omega_{j}$ for some $j$. If $q>3$, then the restriction map 
\[
 \res :\opH^{1}(G,L(\lambda))\rightarrow \opH^{1}(\Gfq,L(\lambda))
\]
is an isomorphism. 
\end{theorem} 

With this isomorphism we are able to compute the cohomology for $\Gfq$ with coefficients in a simple module having fundamental highest weight. 

\begin{theorem} \label{maintheorem2}
Assume that $p>2$ when $\Phi=A_{n}$, $D_{n}$, $p>3$ when $\Phi=B_{n}$, $C_{n}$, $E_{6}$, $E_{7}$, $F_{4}$, $G_{2}$, and $p>5$ when $\Phi=E_{8}$. Suppose $q>3$. Then $\opH^{1}(\Gfq,L(\omega_{j}))=0$ except for the following cases: 
\begin{enumalph} 
\item $\Phi$ has type $C_n$, $n \geq 3$, $(n+1) = \sum_{i=0}^t b_ip^i$ with $0 \leq b_i < p$ and $b_t \neq 0$, and $\lambda = \omega_j$ with $j = 2b_ip^i$ for some $0 \leq i < t$ with $b_i \neq 0$; 
\item $\Phi$ is of type $E_{7}$, $p=7$ and $j=6$. 
\end{enumalph} 
In both cases (a) and (b), we have $\opH^1(\Gfq,L(\lambda)) \cong k$. 
\end{theorem}

Theorem \ref{maintheorem2} is proved for the classical groups in Section \ref{section:classicalgroups}, and for the exceptional groups in Section \ref{section:exceptionalgroups}. Our techniques are also applicable to the case when the simple coefficient module has highest weight less than or equal to a fundamental dominant weight. To obtain complete results, we must enlarge the prime to $p>7$ for the cases $\Phi=E_{7},E_{8}$. 

\begin{theorem} \label{maintheorem3}
Let $\lambda \in X(T)_+$ be such that $\lambda \leq \omega_j$ for some $j$.  Assume that $q>3$ and 
\begin{center}
\begin{tabular}{ll}
$p > 2$ & if $\Phi$ has type $A_n$, $D_n$; \\
$p > 3$ & if $\Phi$ has type $B_{n}$, $C_{n}$, $E_{6}$, $F_{4}$, $G_2$; and \\
$p > 7$ & if $\Phi$ has type $E_7$ or $E_8$.
\end{tabular}
\end{center}
Then $\opH^1(G(\F_q),L(\lambda)) = 0$ except for the following cases, in which $\opH^1(\Gfq,L(\lambda)) \cong k$.
\begin{enumalph}
\item $\Phi$ has type $C_n$, $n \geq 3$, $(n+1) = \sum_{i=0}^t b_ip^i$ with $0 \leq b_i < p$ and $b_t \neq 0$, and $\lambda = \omega_j$ with $j = 2b_ip^i$ for some $0 \leq i < t$ with $b_i \neq 0$.

\item $\Phi$ has type $F_4$, $p = 13$, and $\lambda = 2\omega_4$.

\item $\Phi$ has type $E_7$, $p = 19$, and $\lambda = 2\omega_1$.

\item $\Phi$ has type $E_8$, $p = 31$, and $\lambda = 2\omega_8$.
\end{enumalph}
\end{theorem}

Theorem \ref{maintheorem3} is proved in Section \ref{section:exceptionalgroups}.

\subsection{Definitions and notation} \label{subsection:notation}

Much of the notation used here for algebraic groups is standard and can be found in \cite{Jantzen:2003}. Let $k$ be an algebraically closed field of characteristic $p > 2$, and let $G$ be a simple, simply-connected algebraic group over $k$. Let $T \subset G$ be a maximal torus, defined and split over $\F_p$, and let $\Phi$ be the root system of $T$ in $G$.  Let $\Delta = \set{\alpha_1,\ldots,\alpha_n} \subset \Phi$ be a set of simple roots in $\Phi$, and let $\Phi^+$ and $\Phi^-$ be the corresponding systems of positive and negative roots in $\Phi$. In this paper we use the ordering of the simple roots given in \cite{Humphreys:1978}, following Bourbaki. Let $B \subset G$ be the Borel subgroup of $G$ containing $T$ that corresponds to $\Phi^-$, and let $U \subset B$ be the unipotent radical of $B$. Write $W$ for the Weyl group of $\Phi$, and let $w_0$ be the longest element in $W$.

Let $\E$ be the Euclidean space spanned by $\Phi$. It possesses a $W$-invariant inner product, denoted by $(\cdot,\cdot)$. Given $\alpha \in \Phi$, write $\alpha^\vee = 2\alpha/(\alpha,\alpha)$ for the corresponding coroot. Let $\alpha_0$ be the highest short root in $\Phi$, and set $\rho = \frac1{2} \sum_{\alpha \in \Phi^+} \alpha$. Then the Coxeter number associated to $\Phi$ is $h = (\rho,\alpha_0^\vee)+1$. The weight lattice $X(T)$ is the $\Z$-span in $\E$ of the set of fundamental dominant weights $\set{\omega_1,\ldots,\omega_n}$, which are defined by the equations $(\omega_i,\alpha_j^\vee) = \delta_{i,j}$ (Kronecker delta). Given $\lambda \in X(T)$ and $w \in W$, write $\lambda \mapsto w\lambda$ for the usual action of $W$ on $X(T)$, and write $w \cdot \lambda = w(\lambda + \rho) - \rho$ for the dot action of $W$ on $X(T)$. The weight lattice is partially ordered by the relation $\mu \leq \lambda$ if $\lambda - \mu$ is a nonnegative integral combination of simple roots. Write $X(T)_+$ for the set of dominant weights in $X(T)$, and $X_r(T)$ for the set of $p^r$-restricted dominant weights in $X(T)_+$.

Let $F: G \rightarrow G$ be the Frobenius morphism of $G$. For $r \geq 1$ and $q = p^r$, set $\Gfq = G^{F^r}$, the fixed-point subgroup of $G$ under the $r$-th iterate $F^r: G \rightarrow G$, and set $G_r = \ker F^r$, the scheme-theoretic kernel of the map $F^r: G \rightarrow G$. For $H \subset G$ a closed $F$-stable subgroup (scheme) of $G$, write $H(\Fq) = H^{F^r}$ and $H_r = \ker(F^r|_H: H \rightarrow H)$. Since $T$, $B$ and $U$ are closed $F$-stable subgroups of $G$, there are finite subgroups $\Bfq$, $\Ufq$, and $\Tfq$ of $\Gfq$, and the finite subgroup schemes $B_r$, $U_r$, and $T_r$ of $G_r$.

Set $\fru = \Lie(U)$, the Lie algebra of $U$. Then $\fru$ is a $p$-restricted Lie algebra over $k$, and there exists a $p$-restricted Lie algebra $\ufp$ over $\Fp$, obtained via reduction mod $p$ from a Chevalley basis for $\fru$, such that $\fru = \ufp \ofp k$. Set $\ufq = \ufp \ofp \Fq$. Let $u(\fru)$ be the restricted enveloping algebra of $\fru$. Then $u(\fru)$ is isomorphic to $\Dist(U_1)$, the algebra of distributions on the finite group scheme $U_1$ \cite[I.9.6(4)]{Jantzen:2003}. The category of $U_1$-modules is naturally equivalent to the category of $\Dist(U_1) \cong u(\fru)$-modules \cite[I.8.6]{Jantzen:2003}. Henceforth, given a $u(\fru)$-module (equivalently, a $U_1$-module) $M$, we often identify without further comment the spaces $\opH^\bullet(u(\fru),M)$ and $\opH^\bullet(U_1,M)$.

Write $\modG$ to denote the category of rational $G$-modules. Then any $M \in \modG$ is by restriction also a module for $\Gfq$, $G_r$, $U_r$, etc. Given $\lambda \in X(T)_+$, let $L(\lambda)$ be the simple rational $G$-module of highest weight $\lambda$. If $\lambda \in X_r(T)$, then $L(\lambda)$ remains simple upon restriction to $\Gfq$ and upon restriction to $G_r$ \cite[Theorems 2.5 and 2.11]{Humphreys:2006}.

\section{An analysis of \texorpdfstring{$1$}{1}-cohomology for algebraic groups, finite groups, and Lie algebras} \label{section:ananalysis}

Let $M$ be a finite dimensional rational $G$-module. In this section we relate the first cohomology group $\opH^1(\Gfq,M)$ for $\Gfq$ to the corresponding cohomology groups for the algebraic groups $G$ and $U$ and the Frobenius kernel $U_1$. When $M = L(\lambda)$ with $\lambda$ less than or equal to a fundamental dominant weight, we obtain a vanishing criterion for $\opH^1(\Gfq,M)$ in terms of cohomology for $U_1$. Throughout Section \ref{section:ananalysis}, we will assume that $p$ is {\em excellent} for the root system $\Phi$ (cf. \cite[Section 1.4]{Lin:1999}), that is, $p \neq 2$ when $\Phi=B_{n}$, $C_{n}$, $F_{4}$, and $p > 3$ in type $G_2$. Note that $p$ is excellent whenever the requirements on $p$ stated in the main results are satisfied.

\subsection{Reduction to Sylow \texorpdfstring{$p$}{p}-subgroups} 

The first step in establishing the relationship between the cohomology groups $\opH^1(\Gfq,M)$ and $\opH^1(U_1,M)$ is to consider a suitable subspace of the cohomology for the finite subgroup $\Ufq$ of $\Gfq$. Since $\Ufq$ is a Sylow $p$-subgroup of $\Gfq$, the restriction homomorphism $\opH^\bullet(\Gfq,M) \rightarrow \opH^\bullet(\Ufq,M)$ is injective \cite[Proposition 4.2.2]{Evens:1991}. The torus $\Tfq$ acts on the groups $\Gfq$ and $\Ufq$ by conjugation, and the conjugation actions together with the defining action of $\Tfq$ on $M$ induce actions of $\Tfq$ on $\opH^\bullet(\Gfq,M)$ and $\opH^\bullet(\Ufq,M)$. The restriction map in cohomology is then a homomorphism of $\Tfq$-modules. For any group $G'$ and any $kG'$-module $N$, the inner automorphisms of $G'$ all induce the identity map on $\opH^\bullet(G',N)$ \cite[Proposition 4.1.1]{Evens:1991}. Then $\Tfq$ acts trivally on $\opH^\bullet(\Gfq,M)$, and the restriction homomorphism defines for each $n \geq 0$ an injective map 
\begin{equation} \label{eq:sylowinjection}
\opH^n(\Gfq,M)\ \stackrel{ \res }{\hookrightarrow}\ \opH^n(\Ufq,M)^{\Tfq}.
\end{equation}

\subsection{Weil restriction} \label{subsection:Weilrestriction}

The next step in establishing the relationship between $\Gfq$-cohomology and $U_1$-cohomology is to relate cohomology for the finite group $\Ufq$ to cohomology for a suitable restricted Lie algebra. For this we need the Weil restriction functor constructed by Friedlander.

\begin{definition} \textup{\cite[Definition 1.4]{Friedlander:2010}}
Let $r \geq 1$ and set $q = p^r$. Then the Weil restriction $\Rufq$ of the $\Fq$-Lie algebra $\ufq$ is the $\Fp$-Lie algebra obtained by viewing the underlying $\Fq$-vector space of $\ufq$ as an $\Fp$-vector space, and by viewing the $\Fq$-bilinear bracket on $\ufq$ as an $\Fp$-bilinear map on the underlying $\Fp$-vector space. The Weil restriction $\Rufq$ is made a $p$-restricted Lie algebra by considering the $p$-restriction operator on $\ufq$ as a $p$-restriction operator on the underlying $\Fp$-vector space.
\end{definition}

\begin{proposition} \label{proposition:Weilrestriction} \textup{\cite[Proposition 1.7]{Friedlander:2010}}
Let $r \geq 1$ and set $q = p^r$. Then $\Rufq \ofp k \cong \fru^{\oplus r}$ as $p$-restricted Lie algebras over $k$.
\end{proposition}

\begin{proof}
Write $\ufq \ofp \Fq = (\ufp \ofp \Fq) \ofp \Fq \cong \ufp \ofp (\Fq \ofp \Fq)$. This identification is an isomorphism of $p$-Lie algebras, where the Lie bracket and $p$-operation on $\ufp \ofp (\Fq \ofp \Fq)$ are defined by $[x \otimes a,y \otimes b] = [x,y] \otimes (ab)$ and $(x \otimes a)^{[p]} = x^{[p]} \otimes a^p$. Here $x,y \in \ufp$ and $a,b \in \Fq \ofp \Fq$. Now $\Fq \ofp \Fq \cong (\Fq)^{\times r}$ as an $\Fq$-algebra. To see this, write $\Fq = \Fp[\alpha]$ for some $\alpha \in \Fq$, and let $f \in \Fp[t]$ be the minimal polynomial of $\alpha$ over $\Fp$, so $\deg(f) = r$. Then $\Fq \cong \Fp[t]/(f)$, and $\Fq \ofp \Fq \cong \Fq[t]/(f)$. But $f$ splits over $\Fq$, so $\Fq[x]/(f) \cong (\Fq)^{\times r}$. Now
\[
\ufp \ofp (\Fq \ofp \Fq) \cong \ufp \ofp (\Fq)^{\times r} \cong (\ufq)^{\oplus r}
\]
as $p$-restricted Lie algebras over $\Fq$. Extending scalars to $k$, one gets $\Rufq \ofp k \cong \fru^{\oplus r}$.
\end{proof}

There exists a natural embedding $\iota: \ufp \hookrightarrow \Rufq$ of $\Fp$-Lie algebras corresponding to the fact that $\ufp$ is naturally an $\Fp$-vector subspace of $\ufq$. Explicitly, $\ufp$ identifies with the subspace $\ufp \ofp 1$ of $\ufp \ofp \Fq \cong \Rufq$. Extending scalars to $k$, one obtains an embedding
\[
\iota \ofp k: \fru \cong \ufp \ofp k \longrightarrow \Rufq \ofp k \cong \fru^{\oplus r},
\]
which is just the diagonal embedding of $\fru$ into $\fru^{\oplus r}$. To see this, observe that an element $x \in \ufp \subset \fru$ maps under $\iota \ofp k$ to $x \ofp (1 \ofp 1) \in \ufp \ofp (\Fq \ofp k) \cong \Rufq \ofp k$, and $1 \ofp 1 \in \Fq \ofp k$ is the sum of the $r$ primitive orthogonal idempotents that yield the decomposition $\Fq \ofp k \cong k^{\times r}$. There also exists a surjection $\Rufq \ofp k \cong \ufp \ofp (\Fq \ofp k) \twoheadrightarrow \ufp \ofp k \cong \fru$ induced by the natural multiplication map $\Fq \ofp k \rightarrow k$. This surjection then identifies with the $r$-fold addition map $\fru^{\oplus r} \twoheadrightarrow \fru$. From now on, it will be convenient to denote the Weil restriction $\Rufq$ simply by $\ufq$. Then $u(\ufq \ofp k) \cong u(\fru^{\oplus r})$.

The groups $\Tfp$ and $\Tfq$ act on $\ufp$ and $\ufq$, respectively, by the adjoint action, and the embedding $\ufp \hookrightarrow \ufq$ is a homomorphism of $\Tfp$-modules. Upon scalar extension to $k$, the $\Tfp$-module homomorphism $\ufp \hookrightarrow \ufq$ lifts to a $T$-module homomorphism $\ufp \ofp k \hookrightarrow \ufq \ofp k$, which under the identifications $\ufp \ofp k \cong \fru$ and $\ufq \ofp k \cong \fru^{\oplus r}$ is just the usual adjoint action of $T$.

\subsection{The gr operation}

Set $A = k \Ufq$, the group algebra over $k$ of $\Ufq$, and let $I \subset A$ be the augmentation ideal of $A$. Then the powers of $I$ form a multiplicative filtration of $A$. Set $\gr A = \bigoplus_{i \geq 0} (I^i/I^{i+1})$, the associated graded ring. By \cite[Theorem 2.3]{Lin:1999}, $\gr A$ is isomorphic as a Hopf algebra to $u(\ufq \ofp k)$. The isomorphism is a map of $\Tfq$-modules, where $\Tfq$ acts on $\Ufq$ by conjugation, and the action of $\Tfq$ on $u(\ufq \ofp k)$ is the one described in Section \ref{subsection:Weilrestriction}.

Let $M$ be a $k \Ufq$-module and define $\gr M = \bigoplus_{i \geq 0}M_{i}$ where $M_{i}= I^i.M/I^{i+1}.M$. Then $\gr M$ is naturally a graded module for the graded algebra $\gr A \cong u(\ufq \ofp k)$. We now follow the discussion in \cite[Section 2]{PS:2009}. Let $N$ and $Q$ be $k \Ufq$-modules, and let  
\[
0\rightarrow k \stackrel{\sigma}{\rightarrow} Q \rightarrow N \rightarrow 0
\]
represent a non-split extension in $\Ext^{1}_{\Ufq}(N,k)$, which means that $\im\sigma\subseteq I.Q$. Since $\gr$ takes surjections to surjections, we have by \cite[Section 2]{PS:2009} an extension 
\[
0\rightarrow k \stackrel{\sigma'}\rightarrow \gr Q \rightarrow \gr N \rightarrow 0. 
\]
The resulting extension is non-split because $\im \sigma' \subseteq \bigoplus_{i>0} Q_{i}$. Therefore, we get an injective map 
\[
\Ext^{1}_{\Ufq}(N,k)\stackrel{\gr}{\hookrightarrow} \Ext^{1}_{\gr A}(\gr N,k).
\]
If $N$ is also a $\Bfq$-module then $\gr N$ is a $\gr A \rtimes \Tfq$-module and this map also induces an injection on the space of $\Tfq$ fixed points: 
\begin{equation}
\Ext^{1}_{\Ufq}(N,k)^{\Tfq}\stackrel{\gr}{\hookrightarrow} \Ext^{1}_{\gr A}(\gr N,k)^{\Tfq}.
\end{equation}

\subsection{}

Let $A$ be as in the previous section. If $M$ is a finite dimensional $B$-module, then there exists a (weight) filtration, $M=F^{0}M \supseteq F^{1}M \supseteq \cdots$, defined in \cite[Section 2.4]{Lin:1999}, such that $I^{n}F^{i}M\subseteq F^{i+n}M$. The associated graded module $\grbar M$ is a $\gr A$-module. Note that $\grbar M$ might not coincide with $\gr M$. However, in the cases we consider the two filtrations will give rise to the same module. 

We have an isomorphism of algebras $\gr A \cong u(\ufq \otimes_{\Fp} k) \cong u(\fraku^{\oplus r})$. If $M$ is a rational $B$-module, then the linear isomorphism $M\rightarrow \grbar M$ is an isomorphism of $u(\fraku^{\oplus r})$-modules, where the action on $\grbar M$ is given by $\gr A$, and the action of $u(\fraku^{\oplus r})$ on $M$ is the $k$-linear extension of the restriction of the rational action of $U$ to $\ufq$ regarded as a Lie algebra over $\Fp$ (cf. \cite[Proposition 2.4]{Lin:1999} and \cite[Theorem 4.3]{Friedlander:2010}). Put another way, $u(\fru^{\oplus r})$ acts on $M$ via the surjection $u(\fru^{\oplus r}) \twoheadrightarrow u(\fru)$ discussed in Section \ref{subsection:Weilrestriction} composed with the natural action of $u(\fru)$ on $M$. In particular, the normal subalgebra of $u(\fru^{\oplus r})$ that is isomorphic to $u(\fru)$ and that corresponds to the first component of the direct sum $\fru^{\oplus r}$ acts on $M$ via the natural action of $u(\fru)$ on $M$.

We can now give via Lie algebra cohomology an upper bound for the dimension of $\opH^1(\Gfq,M)$ when $M$ is in $\text{mod}(G)$. 

\begin{theorem} \label{theorem:upperbound}
Let $M$ be a finite dimensional rational $G$-module. Then 
\[
\dim \opH^1(\Gfq,M) \leq \dim \opH^1(u(\fraku^{\oplus r}),M)^{\Tfq}.
\]
\end{theorem}

\begin{proof}
First observe that $\dim \opH^{1}(\Gfq,M)\leq \dim \opH^1(\Ufq,M)^{\Tfq}$ by \eqref{eq:sylowinjection}, so we are left to show that $\dim \opH^1(\Ufq,M)^{\Tfq}\leq \dim \opH^1(u(\fraku^{\oplus r}),M)^{\Tfq}$. Also, recall that $\gr A$ is isomorphic as a Hopf algebra to $u(\ufq \ofp k)$. The isomorphism is a map of $\Tfq$-modules, where $\Tfq$ acts on $\Ufq$ by conjugation, and the action of $\Tfq$ on $u(\ufq \ofp k)$ is described in Section \ref{subsection:Weilrestriction}. 

From \cite[Theorem 3.2]{Lin:1999} we get the May spectral sequence
\begin{equation} \label{eq:Mayspecseq}
E_1^{i,j} = \opH^{i+j}(u(\fraku^{\oplus r}), M)_{(i)} \Rightarrow \opH^{i+j}(\Ufq, M),
\end{equation}
where we are identifying $\gr A$ with $u(\fraku^{\oplus r})$ and $M$ with $\grbar M$. The differentials of \eqref{eq:Mayspecseq} are $\Tfq$-module homomorphisms. The finite group $\Tfq \cong (\F_q^{\times})^n$ is semisimple over $k$, so the fixed point functor $(-)^{\Tfq}$ is exact. Then applying $(-)^{\Tfq}$ to \eqref{eq:Mayspecseq}, we obtain the new spectral sequence
\begin{equation} \label{eq:newMayspecseq}
E_1^{i,j} = \left( \opH^{i+j}(u(\fraku^{\oplus r}), M)_{(i)} \right)^{\Tfq} \Rightarrow \opH^{i+j}(\Ufq,L(\lambda))^{\Tfq}.
\end{equation}
Then $\opH^1(\Ufq,M)^{\Tfq}$ is a subquotient of $\opH^1(u(\fraku^{\oplus r}),M)^{\Tfq}$. In particular,
\[
\dim \opH^1(\Ufq,M)^{\Tfq}\leq \dim \opH^1(u(\fraku^{\oplus r}),M)^{\Tfq}. \qedhere
\]
\end{proof} 

\subsection{}
Let $N$ be a $G$-module, and let $\Gamma$ denote the following composition of maps: 
\begin{equation} \label{eq:GammaDef}
\Ext^{1}_{G(\Fq)}(N,k)\hookrightarrow \Ext^{1}_{\Ufq}(N,k)^{\Tfq}\hookrightarrow \Ext^{1}_{\gr A}(\gr N,k)^{\Tfq}\rightarrow \Ext^{1}_{U_{1}}(\gr N,k)^{\Tfq}
\end{equation} 
For the last map we are identifying $\gr A$ with $u(\fraku^{\oplus r})$, and considering the map in cohomology induced by the inclusion of $u(\fraku)$ into the first component of $u(\fraku^{\oplus r})$ (i.e., induced by the inclusion of $\fraku$ into the first component of $\fraku^{\oplus r}$). We have also identified the cohomology groups for $u(\fraku)$ with those for $U_1$ \cite[I.8.6, I.9.6]{Jantzen:2003}. We will next prove that if $N = L(\lambda)^*$ with $\lambda \in X_1(T)$, then the composition of maps \eqref{eq:GammaDef} fits into a commutative square of first cohomology groups.

\begin{theorem} \label{theorem:relatecoho}
Let $\lambda\in X_{1}(T)$. Then there exists a commutative diagram 
\begin{equation} \label{equation:keycommutativediagram}
\CD 
\opH^{1}(G,L(\lambda)) @> \res >> \opH^{1}(U,L(\lambda))^{T}\\
@V \res VV @V \res VV\\
\opH^{1}(\Gfq,L(\lambda)) @>>\Gamma>  \opH^{1}(U_{1},L(\lambda))^{\Tfq}
\endCD
\end{equation}
where $\Gamma$ is obtained by setting $N=L(\lambda)^{*}$ in \eqref{eq:GammaDef}. Furthermore, if 
$\lambda\leq \omega_{j}$ and $q>3$, then $\Gamma$ is injective.
\end{theorem} 

\begin{proof} First observe that if $M$ is a rational $B$-module such that $M$ is generated as a $u(\fraku)$-module by a highest weight vector, then $\gr M=\grbar M$ as a $u(\fraku^{\oplus r})$-module. This can be seen by analyzing the action of root subgroups (as in the proof of \cite[Proposition 2.4]{Lin:1999}) to show that the (weight) filtration coincides with the radical filtration on $M$. In particular, this applies when $M=L(\lambda)$ with $\lambda\in X_{1}(T)$. 

The commutativity of the diagram reduces to proving that if 
\[
0\rightarrow k \rightarrow Q \stackrel{\phi}{\rightarrow} L(\lambda) \rightarrow 0
\]
is a nonsplit extension of rational $B$-modules, then the extension of $\gr A$-modules 
\[
0\rightarrow k \rightarrow \gr Q \rightarrow  \gr L(\lambda) \rightarrow 0
\]
is equivalent to the extension obtained via restriction (to $\ufq$) of $u(\fraku^{\oplus r})$-modules. 

Let $\phi:Q\rightarrow L(\lambda)$ be the map of $B$-modules given above. Let $\phi_{1}:Q\rightarrow L(\lambda)$ denote the restriction of $\phi$ by considering $Q$ and $L(\lambda)$ as $u(\fraku^{\oplus r})$-modules, and let $\phi_{2}:\gr Q \rightarrow \gr L(\lambda)$ be the induced map of $\gr A$-modules, which is also a surjection. By the preceding paragraph we have $\gr Q=\grbar Q$ (because $Q$ is generated by 
a highest weight vector, namely, the inverse image under $\phi$ of a highest weight vector in $L(\lambda)$, or else the sequence splits) and $\gr L(\lambda)=\grbar L(\lambda)$.  Let $\delta_{1}:Q\rightarrow \gr Q$ (resp.\ $\delta_{2}: L(\lambda) \rightarrow \gr L(\lambda)$) be the isomorphism described in the preceding section. By checking on weight spaces (cf. \cite[Proposition 2.4]{Lin:1999}) one can show that we have a commutative diagram of $\gr A$-modules: 
\begin{equation}
\CD 
Q@>\phi_{1}>> L(\lambda) \\
@V\delta_{1}VV @V\delta_{2}VV\\
\gr Q @>>\phi_{2}>  \gr L(\lambda). 
\endCD
\end{equation} 
This proves the equivalence of the extensions. 

Recall that $\Gamma$ is a composition of maps. All the maps are injective except possibly the restriction map: 
\[
\text{res}:\opH^{1}(u(\fraku^{\oplus r}),L(\lambda))^{\Tfq}\rightarrow \opH^{1}(u(\fraku),L(\lambda))^{\Tfq}.
\]
We shall prove that this map is injective under the assumption that $\lambda\leq \omega_{j}$ and $q>3$. The statement clearly holds for $r=1$, in which case res is the identity, so we can assume that $r>1$. 

Consider now the Lyndon--Hochschild--Serre (LHS) spectral sequence for $u(\ufq \ofp k) \cong u(\fru^{\oplus r})$ and its normal subalgebra $u(\fru)$ (i.e., the subalgebra corresponding to the first component of $\fru^{\oplus r}$):
\[
E_2^{i,j} = \opH^i(u(\fru^{\oplus (r-1)}),\opH^j(u(\fru),L(\lambda))) \Rightarrow \opH^{i+j}(u(\fru^{\oplus r}), L(\lambda)).
\]
As before, we can take $\Tfq$-invariants to obtain a new spectral sequence with $E_2$-page $(E_2^{i,j})^{\Tfq}$. Then the $5$-term exact sequence of the new spectral sequence has initial terms
\[
0 \rightarrow (E_2^{1,0})^{\Tfq} \rightarrow \opH^1(u(\fru^{\oplus r}),L(\lambda))^{\Tfq} \rightarrow (E_2^{0,1})^{\Tfq}.
\]
Here $(E_2^{0,1})^{\Tfq}$ identifies with a subspace of $\opH^1(u(\fru),L(\lambda))^{\Tfq}$. 

It remains to show that $(E_2^{1,0})^{\Tfq} = 0$. We have
\[
\opH^0(u(\fru), L(\lambda)) \cong \Hom_{u(\fru)}(k, L(\lambda)) \cong \Hom_{U_1}(k,L(\lambda)) \cong w_0 \lambda
\]
as a $T$-module because $\lambda\in X_{1}(T)$ \cite[II.3.12]{Jantzen:2003}. Then $E_2^{1,0} \cong \opH^1(u(\fru^{\oplus (r-1)}),k) \otimes w_0 \lambda$, where we have pulled out the weight $w_0\lambda$ because it is trivial as a module for $u(\fru^{\oplus (r-1)})$. Since $u(\fru^{\oplus (r-1)}) \cong u(\fru)^{\otimes (r-1)}$, we have $\opH^{\bullet}(u(\fru^{\oplus (r-1)}),k) \cong \opH^\bullet(u(\fru),k)^{\otimes (r-1)} \cong \opH^\bullet(U_1,k)^{\otimes (r-1)}$ by \cite[Theorem X.7.4]{Mac-Lane:1995}. In particular, $\opH^1(u(\fru^{\oplus (r-1)}),k) \cong \bigoplus_{i=1}^{r-1} \opH^1(U_1,k)$. Now $(E_2^{1,0})^{\Tfq} \neq 0$ only if there exists a weight $\beta$ of $T$ in $\opH^1(U_1,k)$ such that $\beta + w_0\lambda \in (q-1)X(T)$. If $\beta$ is a weight of $T$ in $\opH^1(U_1,k)$, then $\beta \in \Delta$ by Lemma \ref{lemma:H1weights} below. But if $\beta \in \Delta$ and $\lambda\leq \omega_{j}$, then $\beta + w_0\lambda \notin (q-1)X(T)$ whenever $q>3$. Hence, $(E_2^{1,0})^{\Tfq} = 0$, and consequently $\Gamma$ is injective.
\end{proof}

\subsection{}

The injectivity of the map $\Gamma$ allows us to state the following vanishing result that will be used throughout the paper. 

\begin{corollary} \label{corollary:vanishingresult}
Let $L(\lambda)$ be a simple $G$-module with $\lambda\leq \omega_{j}$ and suppose $q>3$. Then
\begin{enumalph} 
\item $\dim \opH^1(\Gfq,L(\lambda)) \leq \dim \opH^1(U_1,L(\lambda))^{\Tfq}$, and
\item if $\opH^1(U_1,L(\lambda))^{\Tfq}=0$, then $\opH^1(\Gfq,L(\lambda))=0$.
\end{enumalph} 
\end{corollary}

\begin{remark} \label{remark:lambda0}
When $\lambda=0$, we have $\opH^{1}(G,k)= \opH^1(B,k) = \opH^{1}(\Gfq,k)=\opH^{1}(U_{1},k)^{\Tfq}=0$. Indeed $\opH^{1}(G,k) \cong \opH^1(B,k) = 0$ by \cite[II.4.11]{Jantzen:2003}. By Lemma \ref{lemma:H1weights} below, the weights of $\opH^{1}(U_{1},k)$ are simple roots, and none of these is $\Tfq$-invariant if $q > 3$. (The only $T(\F_3)$-invariant simple roots occur when $\Phi$ is of type $A_1$ or $B_2$.) The vanishing of $\opH^{1}(\Gfq,k)$ now follows from Corollary \ref{corollary:vanishingresult}(b). The vanishing of $\opH^1(\Gfq,k)$ can also be proved directly by an argument using the Frattini subgroup. 
\end{remark}

Because of the above remark, we may henceforth restrict our attention to $\lambda\ne 0$.

\section{Cohomology for the Frobenius kernel \texorpdfstring{$U_1$}{U1}} \label{section:U1cohomology}

In this section we study the cohomology group $\opH^1(U_1,L(\lambda))$ of Theorem \ref{theorem:relatecoho}. Eventually we will specialize to the case where $\lambda$ is less than or equal to a fundamental dominant weight. Throughout this section, we maintain the standing assumption that $p > 2$.

\subsection{Weight spaces in the socle of \texorpdfstring{$U_1$}{U1} cohomology} \label{subsection:Andersenanalysis}

Given $\lambda \in X(T)$, set $\lambda^* = -w_0\lambda$. Observe that the involution $\lambda \mapsto \lambda^*$ restricts to involutions on $\Delta$ and $X_1(T)$. In particular, if $\lambda \in X_1(T)$ is less than or equal to a fundamental dominant weight, then so is $\lambda^*$. Also, $\Ext_{U_1}^1(L(\lambda),k) \cong \Ext_{U_1}^1(k,L(\lambda^*)) = \opH^1(U_1,L(\lambda^*))$, so understanding the $\Tfq$-invariants in $\Ext_{U_1}^1(L(\lambda),k)$ will enable us to apply Theorem \ref{theorem:relatecoho} to study cohomology for $\Gfq$. Our first step is to analyze the socle of $\Ext^1_{U_1}(L(\lambda),k)$ as a $B/U_1 \cong (U/U_1) \rtimes T$-module.

Every simple rational $B/U_1$-module is one-dimensional of $T$-weight $-\mu - p\nu$ for some $\mu \in X_1(T)$ and some $\nu \in X(T)$. The dimension of the $(-\mu - p\nu)$-isotypic component in the socle of a rational $B/U_1$-module $M$ is equal to $\dim \Hom_{B/U_1}(-\mu-p\nu,M)$. Then to compute the socle of $\Ext_{U_1}^1(L(\lambda),k)$ as a $B/U_1$-module, it suffices to consider, for $\mu \in X_1(T)$ and $\nu \in X(T)$, the dimensions of the $\Hom$-spaces
\begin{equation} \label{eq:soclereduction}
\begin{split}
\Hom_{B/U_1} (-\mu-p\nu , \Ext^1_{U_1}(L(\lambda),k)) &\cong \Hom_{B/U_1} (k, \Ext^1_{U_1}(L(\lambda),k)\otimes (\mu+p\nu)) \\
&\cong \Hom_{B/U_1}(k, \Ext_{U_1}^1(L(\lambda),\mu+p\nu)) \\
&\cong \Hom_{B/B_1}(k, \Hom_{T_1}(k,\Ext^1_{U_1}(L(\lambda),\mu+p\nu))) \\
&\cong \Hom_{B/B_1}(k, \Ext_{B_1}^1(L(\lambda),\mu+p\nu)).
\end{split}
\end{equation}
Here we have used the fact that $T_1\cong B_1/U_1$ is a normal subgroup scheme in $B/U_1$ with quotient $B/B_1$. The last isomorphism follows by applying the LHS spectral sequence for the group extension $1 \rightarrow U_1 \rightarrow B_1 \rightarrow T_1 \rightarrow 1$ and using the fact that modules over $T_1\cong B_1/U_1$ are completely reducible.

\begin{lemma} \label{lemma:nonzeroweightspace}
Suppose $p > 2$. Let $\lambda,\mu \in X_1(T)$, and let $\nu \in X(T)$. Then
\[
\Hom_{B/U_1} (-\mu-p\nu , \Ext^1_{U_1}(L(\lambda),k)) \cong \begin{cases} \Ext^1_B(L(\lambda), \mu + p\nu) & \text{if $\lambda \neq \mu$} \\ 0 & \text{if $\lambda = \mu$.} \end{cases}
\]
\end{lemma}

\begin{proof}
First suppose that $\lambda = \mu$. Since $B_1$ acts trivially on $p\nu$, one has
\[
\Ext^1_{B_1} (L(\lambda),\lambda+p\nu) \cong \Ext^1_{B_1} (L(\lambda),\lambda) \otimes p\nu = 0
\]
by \cite[Theorem 3.4]{Andersen:1984a}. Then $\Hom_{B/U_1} (-\mu-p\nu , \Ext^1_{U_1}(L(\lambda),k)) = 0$ by \eqref{eq:soclereduction}. So assume that $\lambda \neq \mu$, and consider the LHS spectral sequence
\[
E_2^{i,j} = \Ext^i_{B/B_1} (k , \Ext^j_{B_1} (L(\lambda), \mu + p\nu)) \Rightarrow \Ext^{i+j}_B (L(\lambda), \mu + p\nu).
\]
It gives rise to the $5$-term exact sequence
\begin{equation} \label{eq:B5term}
0 \rightarrow E_2^{1,0} \rightarrow \Ext_B^1(L(\lambda),\mu+p\nu) \rightarrow E_2^{0,1} \rightarrow E_2^{2,0}\rightarrow E_2.
\end{equation}
One has $\Hom_{B_1}(L(\lambda),\mu+p\nu) \cong \Hom_{B_1}(L(\lambda), \mu) \otimes p\nu = 0$ because $\lambda\neq \mu$. (Since $\lambda,\mu \in X_1(T)$, any homomorphism must map the highest weight space $L(\lambda)_\lambda$ to zero because $\lambda$ and $\mu$ are not congruent modulo $pX(T)$, and $L(\lambda)$ is generated as a $U_1$-module by its highest weight space \cite[II.3.14]{Jantzen:2003}.) Then $E_2^{1,0}=E_2^{2,0}=0$, and $\Hom_{B/U_1} (-\mu-p\nu , \Ext^1_{U_1}(L(\lambda),k)) \cong E_2^{0,1} \cong \Ext^1_B(L(\lambda), \mu + p\nu)$.
\end{proof}

\begin{remark}
One may have $\Ext^1_B(L(\lambda), \mu + p\nu) \neq 0$ but $\Hom_{B/U_1} (-\mu-p\nu , \Ext^1_{U_1}(L(\lambda),k)) = 0$. Indeed, let $\lambda = \mu \in X_1(T)$, let $\alpha \in \Delta$, and take $\nu = -\alpha$. Then $\Ext_{B_1}^1(L(\lambda),\mu+p\nu) = 0$ as in the proof of the lemma, so in \eqref{eq:B5term} one has $E_2^{0,1} = 0$. Then $E_2^{1,0} \cong \Ext_B^1(L(\lambda),\mu+p\nu)$. Now
\begin{align*}
E_2^{1,0} &\cong \Ext_{B/B_1}^1(k,\Hom_{B_1}(L(\lambda),\lambda) \otimes p\nu) \\
&\cong \Ext_{B/B_1}^1(k,p\nu) & \text{because } \Hom_{B_1}(L(\lambda),\lambda) \cong k, \\
&\cong \Ext_B^1(k,-\alpha) \cong k & \text{by \cite[Corollary 2.4]{Andersen:1984a}.}
\end{align*}
\end{remark}

\begin{corollary} \label{corollary:BtoU1injection}
Suppose $p > 2$, and let $\lambda \in X_1(T)$. Then the restriction map $\Ext^{1}_{B}(L(\lambda),k) \to \Ext^{1}_{U_{1}}(L(\lambda),k)^{\Tfq}$ is an injection.
\end{corollary}

\begin{proof}
If $\lambda=0$, then $\Ext_B^1(L(\lambda),k)=0$ by Remark \ref{remark:lambda0}, so assume $\lambda\ne 0$. By the proof of Lemma \ref{lemma:nonzeroweightspace} (with $\mu=\nu=0$), we have $\Ext^{1}_{B}(L(\lambda),k) \cong \Ext^{1}_{U_{1}}(L(\lambda),k)^{B/U_{1}} \subseteq \Ext^{1}_{U_{1}}(L(\lambda),k)^{\Tfq}$.
\end{proof}
 
\subsection{Structure of the socle}

We can now describe the socle of $\Ext^1_{U_1}(L(\lambda),k)$ as a $B/U_1$-module. Given weights $\lambda,\mu \in X(T)$, write $\lambda \uparrow \mu$ for the order relation on $X(T)$ defined in \cite[II.6.4]{Jantzen:2003}.

\begin{theorem} \label{theorem:socthm}
Suppose $p > 2$, and let $\lambda \in X_1(T)$. Then
\[
\soc_{B/U_1}\Ext^1_{U_1}(L(\lambda),k) \cong \bigoplus_{\substack{\alpha\in \Delta \\ (\lambda,\alpha^\vee) \neq p-1}} -s_{\alpha}\cdot \lambda 
\oplus \bigoplus_{\substack{\sigma \uparrow \lambda \\ \sigma\in X(T)_{+}}} (-\sigma)^{\oplus m_\sigma},
\]
where $m_\sigma=\dim \Ext^1_G(L(\lambda),H^0(\sigma))$.
\end{theorem}

\begin{proof}
Let $\mu \in X_1(T)$ with $\mu \neq \lambda$, and let $\nu \in X(T)$. Then by Lemma \ref{lemma:nonzeroweightspace}, the simple summands in $\soc_{B/U_1}\Ext^1_{U_1}(L(\lambda),k)$ have the form $-\mu-p\nu$, and occur with multiplicity $\dim \Ext^1_B(L(\lambda), \mu + p\nu)$. So suppose $-\mu-p\nu$ occurs as a summand in $\soc_{B/U_1}\Ext^1_{U_1}(L(\lambda),k)$. First suppose $\mu + p\nu \in X(T)_+$, so also $\nu \in X(T)_+$. Then by Kempf's vanishing theorem \cite[II.4.5]{Jantzen:2003} and \cite[I.4.5]{Jantzen:2003},
\[
\Ext^1_B(L(\lambda),\mu+p\nu)\cong \Ext^1_G(L(\lambda), H^0(\mu + p\nu)).
\]
This space is non-zero by assumption, so $\mu+p\nu \uparrow \lambda$ by \cite[II.6.20]{Jantzen:2003}. Conversely, let $\sigma \in X(T)_+$ with $\sigma \uparrow \lambda$. Then $\sigma \leq \lambda$. If $\sigma = \lambda$, then $m_\sigma = 0$ by \cite[II.6.20]{Jantzen:2003}, so assume that $\sigma \neq \lambda$. Write $\sigma = \mu + p \nu$ with $\mu \in X_1(T)$ and $\nu \in X(T)_+$. If $\lambda = \mu$, we would have $\lambda + p\nu \leq \lambda$, and hence $p\nu \leq 0$, a contradiction, because $0 \neq p \nu \in X(T)_+$ and every dominant weight is a positive rational combination of simple roots. Thus, we conclude for all $\sigma \in X(T)_+$ with $\sigma \uparrow \lambda$ that $-\sigma$ occurs as a summand in $\soc_{B/U_1}\Ext^1_{U_1}(L(\lambda),k)$ with multiplicity $m_\sigma$.

Now suppose that $\mu + p\nu \notin X(T)_+$. Then from the first case of \cite[Proposition 2.3]{Andersen:1984a} we conclude that $\mu +p\nu = s_\alpha \cdot \lambda$ for some $\alpha \in \Delta$, and that $\dim \Ext^{1}_{B}(L(\lambda), s_{\alpha} \cdot \lambda) = 1$. Observe that the second case of the cited proposition cannot occur here because, by assumption, $\mu, \lambda \in X_{1}(T)$ and $\mu \neq \lambda$. If $\mu + p\nu = s_\alpha \cdot \lambda$ and $\mu \neq \lambda$, then necessarily $(\lambda,\alpha^\vee) \neq p-1$. Conversely, if $(\lambda,\alpha^\vee) \neq p-1$, then $s_\alpha \cdot \lambda$ has the form $\mu + p\nu$ with $\mu \in X_1(T)$, $\nu \in X(T)$, and $\mu \neq \lambda$. Thus, we conclude for all $\alpha \in \Delta$ with $(\lambda,\alpha^\vee) \neq p-1$ that $-s_\alpha \cdot \lambda$ occurs once as a summand in $\soc_{B/U_1}\Ext^1_{U_1}(L(\lambda),k)$.
\end{proof}

\begin{corollary} \label{corollary:socle}
Suppose $p > 3$ if $\Phi$ is of type $E_7$, $E_8$, or $F_4$, and $p > 2$ otherwise. Let $\lambda \in X(T)_+$ with $\lambda \leq \omega_j$ for some $j$. Then
\[
\soc_{B/U_1}\Ext^1_{U_1}(L(\lambda),k) \cong \bigoplus_{\alpha\in \Delta} -s_{\alpha}\cdot \lambda 
\oplus \bigoplus_{\substack{\sigma \uparrow \lambda \\ \sigma\in X(T)_{+}}} (-\sigma)^{\oplus m_\sigma},
\]
\end{corollary}

\begin{proof}
One can verify from the lists in Section \ref{subsection:Hassediagrams} that for all $\alpha \in \Delta$, $(\lambda,\alpha^\vee) \neq p-1$.
\end{proof}

It is interesting to note that the contribution to the socle of $\Ext^1_{U_1}(L(\lambda),k)$ comes from two sources. The factors in the first direct summand seem to come from Kostant's classical theorem for the cohomology of complex semisimple Lie algebras, while for $p$ large, the multiplicities $\dim \Ext^1_G(L(\lambda),H^0(\sigma))$ of the other factors arise as coefficients of Kazhdan-Lusztig polynomials. In the special case when the Weyl module $V(\lambda)$ is simple, we obtain the following corollary. 

\begin{corollary} \label{corollary:corsoc}
Suppose $p > 2$ and $\lambda\in X_{1}(T)$.  If $L(\lambda) = V(\lambda),$ then
\[
\soc_{B/U_1} \Ext^1_{U_1}(L(\lambda), k) = 
\bigoplus_{\alpha \in \Delta} -s_\alpha \cdot \lambda.
\]
\end{corollary}

\begin{proof} Apply the vanishing result $\Ext^1_G(V(\lambda),H^0(\sigma))=0$ of \cite[II.4.13]{Jantzen:2003} to Corollary \ref{corollary:socle}.
\end{proof}

\subsection{Constraints on weights} \label{subsection:injectivehullweights}

Next we examine the structure of $M := \Ext^1_{U_1}(L(\lambda),k)$ as a $B/U_1$-module. We are interested in cases for which the socle of $M$ is equal to the entire module.

\begin{lemma} \label{lemma:H1weights}
Let $V$ be a finite dimensional rational $B$-module. Let $\mu$ be a weight of $T$ in $\Ext_{U_1}^1(V,k) \cong \opH^1(U_1,V^*)$. Then $\mu = \beta - \nu$ for some $\beta \in \Delta$ and some weight $\nu$ of $V$.
\end{lemma}

\begin{proof}
First, if $\mu$ is a weight of $T$ in $\opH^1(U_1,V^*)$, then $\mu$ is also a weight of $T$ in $\opH^1(U_1,k) \otimes V^*$ by the argument in \cite[\S 2.5]{UGA:2009}. Also, the weights of $V^*$ are precisely $\set{-\nu: V_\nu \neq 0}$. Next, write $\Dist(U_1)_+$ for the augmentation ideal of the algebra $\Dist(U_1)$. By inspecting the low degree terms in the cobar resolution computing $\opH^\bullet(U_1,k) = \opH^\bullet(\Dist(U_1),k)$, one sees that $\opH^1(U_1,k)$ is a $T$-module subquotient of the space $\Hom_k(\Dist(U_1)_+/(\Dist(U_1)_+)^2,k)$. The $T$-weights of the latter space are precisely the simple roots in $\Delta$.
\end{proof}

For any $\mu\in X(T)$, consider the injective hull $I(\mu)$ of $\mu$ in the category of rational $B/U_1$-modules. Set $Q = \bigoplus_{\alpha \in \Delta} I(-s_\alpha \cdot \lambda) \oplus \bigoplus_{\sigma \uparrow \lambda} I(-\sigma)^{\oplus m_\sigma}$. We have $\soc_{B/U_1} M = \soc_{B/U_1} Q$ by Corollary \ref{corollary:socle}, so there exists an injection $M \hookrightarrow Q$. To show that $\soc_{B/U_1}M = M$, it suffices to show that no weight from the second socle layer of $Q$ can be a weight of $M$. (Recall that the second socle layer of $Q$ is defined as $\soc_{B/U_1} (Q/\soc_{B/U_1} Q)$.) For $\mu \in X$, one has $I(\mu) \cong k[U/U_1] \otimes \mu$ as a $B/U_1$-module by \cite[I.3.11]{Jantzen:2003}, where $k[U/U_1]$ denotes the coordinate ring of the unipotent group $U/U_1$.

\begin{lemma} \label{lemma:secondsocle}
The second socle layer of the $B/U_1$-module $I(\mu)$ consists of one-dimensional modules of the form $\mu + p^m \gamma$ with $\gamma \in \Delta$ and $m > 0$.
\end{lemma}

\begin{proof}
It suffices to describe the second socle layer of the $B/U_1$-module $I(0) \cong k[U/U_1]$. Let
\[
0 \rightarrow k \stackrel{\varepsilon}{\rightarrow} I(0) \stackrel{d_0}{\rightarrow} I_1 \rightarrow I_2 \rightarrow \cdots
\]
be a minimal injective resolution of the $B/U_1$-module $k$. Then $\soc I_1 \cong \soc (I(0)/\soc I(0))$. Also, for all $\nu \in X(T)$ and all $i \geq 0$, one has $\Hom_{B/U_1}(\nu,I_i) \cong \Ext_{B/U_1}^i(\nu,k) \cong \opH^i(B/U_1,-\nu)$. In particular, the weight $\nu$ occurs in the second socle layer of $I(0)$ with multiplicity $\dim \Hom_{B/U_1}(\nu,I_1) = \dim \opH^1(B/U_1,-\nu)$. Now consider the LHS spectral sequence
\[
E_2^{i,j} = \opH^i(B/U_1,\opH^j(U_1,-\nu)) \Rightarrow \opH^{i+j}(B,-\nu).
\]
It gives rise to the 5-term exact sequence
\[
0 \rightarrow \opH^1(B/U_1,-\nu) \rightarrow \opH^1(B,-\nu) \rightarrow E_2^{0,1} \rightarrow E_2^{2,0} \rightarrow E_2.
\]
By \cite[Corollary 2.4]{Andersen:1984a}, $\opH^1(B,-\nu) = 0$ unless $\nu = p^m \gamma$ for some $\gamma \in \Delta$ and some $m \geq 0$. Then the weight $\nu$ occurs in the second socle layer of $I(0)$ only if $\nu = p^m \gamma$ for some $\gamma \in \Delta$ and some $m \geq 0$. But the weights of $T$ in $I(0) \cong k[U/U_1] \cong k[U^{(1)}]$ are all divisible by $p$, so if $\nu = p^m \gamma$ occurs in the second socle layer of $I(0)$, then necessarily $m > 0$.
\end{proof}

Set $Q_{\text{Kos}} = \bigoplus_{\alpha \in \Delta} I(-s_\alpha \cdot \lambda)$, a submodule of $Q$. By Lemmas \ref{lemma:H1weights} and \ref{lemma:secondsocle}, if the second socle layer of $Q_{\text{Kos}}$ contains a vector of the same weight as a vector in $M$, then $-s_\alpha \cdot \lambda + p^m \gamma = \beta - \nu$ for some $\alpha,\beta,\gamma \in \Delta$, some $m > 0$, and some weight $\nu$ of $L(\lambda)$. Equivalently,
\begin{equation} \label{eq:magicequation}
\lambda - \nu = -\beta+ (\lambda + \rho, \alpha^\vee ) \alpha + p^m \gamma.
\end{equation}
Since $\nu \leq \lambda$, the right-hand-side of \eqref{eq:magicequation} must be an element of $\N \Phi^+$. Since $\alpha,\beta,\gamma$ are simple roots, this implies that $\beta \in \set{\alpha,\gamma}$.

Given $J \subseteq \Delta$, let $H_J^0(\lambda)$ be the induced module of highest weight $\lambda$ for the standard Levi subgroup $L_J$ of $G$; see \cite[II.5.21]{Jantzen:2003}. Then
\begin{equation} \label{eq:Leviinduced}
H^0_J(\lambda) = \bigoplus_{\theta \in \N J} H^0(\lambda)_{\lambda - \theta}
\end{equation}
by \cite[II.5.21]{Jantzen:2003}. Suppose equation \eqref{eq:magicequation} holds. Then 
\begin{equation} \label{eq:rank2constraints}
\nu = \lambda - [ -\beta + (\lambda + \rho,\alpha^\vee)\alpha + p^m\gamma].
\end{equation} 
Since $\nu$ is a weight of $L(\lambda)$, and hence also of $H^0(\lambda)$, we conclude from \eqref{eq:Leviinduced} that $\nu$ must be a weight of $H^0_{J}(\lambda)$ for some $J \subseteq \Delta$ with $\abs{J} \leq 2$.

Now set $Q_{\text{KL}} = \bigoplus_{\sigma \uparrow \lambda} I(-\sigma)^{\oplus m_\sigma}$. As before, if the second socle layer of $Q_{\text{KL}}$ contains a vector of the same weight as a vector in $M$, then
\begin{equation} \label{eq:sigmamagicequation}
-\sigma + p^m \gamma = \beta - \nu
\end{equation}
for some dominant weight $\sigma \in X(T)_+$ with $\sigma \uparrow \lambda$, some $\gamma,\beta \in \Delta$, some integer $m > 0$, and some weight $\nu$ of $L(\lambda)$. Taking the inner product with the dual root $\gamma^\vee$, we get 
\begin{equation} \label{eq:innerproductwithgammavee}
(-\nu,\gamma^\vee)+(\sigma,\gamma^\vee)=p^{m}(\gamma,\gamma^\vee)-(\beta,\gamma^\vee)=2p^m-(\beta,\gamma^\vee),
\end{equation}
and hence
\begin{equation} \label{eq:KLconstraint}
(-\nu,\gamma^\vee)+(\sigma,\gamma^\vee)\geq 2p^m-2. 
\end{equation}

\subsection{Semisimplicity of \texorpdfstring{$U_1$}{U1} cohomology}

We can now give conditions under which the $\Ext$-group $\Ext_{U_1}^1(L(\lambda),k)$ is semisimple as a $B/U_1$-module, that is, $\Ext_{U_1}^1(L(\lambda),k) = \soc_{B/U_1} \Ext_{U_1}^1(L(\lambda),k)$.

\begin{theorem} \label{theorem:U1-coho} Let $\lambda \in X(T)_+$ with $\lambda \leq \omega_j$ for some fundamental weight $\omega_j$. Assume that
\begin{center}
\begin{tabular}{ll}
$p > 2$ & if $\Phi$ has type $A_n$, $D_n$; \\
$p > 3$ & if $\Phi$ has type $B_n$, $C_n$, $E_6$, $E_7$, $F_4$;  \\
$p > 5$ & if $\Phi$ has type $E_{8}$ or $G_{2}$.
\end{tabular}
\end{center}
Then as a $B/U_1$-module,
\[
\Ext^1_{U_1}(L(\lambda), k) = \soc_{B/U_1} \Ext_{U_1}^1(L(\lambda),k) = \bigoplus_{\alpha \in \Delta} -s_\alpha \cdot \lambda \oplus \bigoplus_{\sigma\uparrow \lambda} (-\sigma)^{\oplus m_\sigma},
\]
where $m_\sigma=\dim \Ext^1_G(L(\lambda),H^0(\sigma))$. 
\end{theorem}

\begin{proof}
We show that no weight in the second socle layer of $Q = Q_{\text{Kos}} \oplus Q_{\text{KL}}$ can be a weight of $M$. First suppose that a weight from the second socle layer of $Q_{\text{Kos}}$ is a weight of $M$. Then by the discussion in Section \ref{subsection:injectivehullweights}, there exists a subset $J = \set{\alpha,\beta,\gamma} \subseteq \Delta$ with $\abs{J} \leq 2$, an integer $m > 0$, and a weight $\nu$ of $H_J^0(\lambda)$ such that $\lambda - \nu = -\beta +(\lambda+\rho,\alpha^\vee)\alpha + p^m\gamma$. In Section \ref{subsection:fundamentalweights} we consider the restriction of $\lambda$ to all root subsystems of $\Phi$ of rank $\leq 2$, and in each case compute all possible values for $\lambda - \nu$. An elementary case-by-case analysis shows that, under the stated restrictions on $p$, the equation $\lambda - \nu = -\beta +(\lambda+\rho,\alpha^\vee)\alpha + p^m\gamma$ has no solutions, so no weight in the second socle layer of $Q_{\text{Kos}}$ can be a weight of $M$.

Now suppose that a weight from the second socle layer of $Q_{\text{KL}}$ is a weight of $M$. Note that $Q_{\text{KL}} = \{ 0 \}$ in types $A_{n}$ and $D_{n}$ by Corollary \ref{corollary:corsoc}, because in these types $\lambda\le\omega_{j}$ implies $\lambda=\omega_{i}$ for some $i$, and $L(\omega_{i})=V(\omega_{i})$ (cf.\ Section \ref{section:ABD}). So we may assume that $\Phi$ is not type $A_{n}$ or $D_{n}$.
Then by \eqref{eq:sigmamagicequation}, there exist simple roots $\beta,\gamma \in \Delta$, a dominant weight $\sigma \in X(T)_+$ with $\sigma \uparrow \lambda$, an integer $m > 0$, and a weight $\nu$ of $L(\lambda)$ such that $-\sigma + p^m \gamma = \beta - \nu$.  Assume for the moment that $\Phi$ is not of type $E_{8}$. Then by \eqref{eq:KLconstraint} and the stated assumption on $p$,
\begin{equation} \label{eq:KLinequality}
(-\nu,\gamma^\vee)+(\sigma,\gamma^\vee)\geq 8.
\end{equation}
Since $\sigma \uparrow \lambda$ implies $\sigma \leq \lambda$, one gets $(\sigma,\gamma^\vee) \in \set{0,1,2,3}$ from the list of possible values for $\sigma$ in the Appendix (Section \ref{subsection:Hassediagrams}). Then necessarily $(-\nu,\gamma^\vee)\geq 5$.

Choose $w\in W$ such that $\wt{\alpha} := w\gamma$ is dominant (so $\widetilde\alpha$ will be either the highest short root or the highest long root). Then, recalling that $\lambda^{*}=-w_{0}\lambda$ and $-w_{0}\omega_{j}=\omega_{i}$ for some $i$, 
\begin{equation} \label{eq:nuinequality}
(-\nu,\gamma^\vee)=(-w\nu,w\gamma^\vee)=(-w\nu,\widetilde\alpha^\vee) \leq (-w_0\lambda,\widetilde\alpha^\vee)\leq (-w_0\omega_j,\widetilde\alpha^\vee) = (\omega_i,\widetilde\alpha^\vee),
\end{equation}
which is strictly less than 5, and leads to a contradiction. Finally, suppose $\Phi$ is of type $E_8$ with $p>5$. Then the previous argument leads to a contradiction, because \eqref{eq:KLinequality} becomes $(-\nu,\gamma^\vee)+(\sigma,\gamma^\vee)\geq 12$, whereas we still have $(\sigma,\gamma^\vee) \le 3$, and \eqref{eq:nuinequality} demonstrates that $(-\nu,\gamma^\vee)\leq 6$. 
\end{proof}

\section{Applications} \label{section:applications}

\subsection{An isomorphism with \texorpdfstring{$\mathbf{G}$}{G} cohomology} \label{subsection:isowithG}

The results of Section \ref{section:U1cohomology} enable us to give conditions under which the restriction map $\opH^1(G,L(\lambda)) \rightarrow \opH^1(\Gfq,L(\lambda))$ is an isomorphism, and hence allow us to prove Theorem \ref{maintheorem1}. For the sake of smoothness of exposition, we handle the case of $\Phi$ of type $G_2$ first as a separate result, because some subtleties arise there when treating the case $p=5$. Our proof of Theorem \ref{maintheorem1} for type $G_2$ also establishes Theorems \ref{maintheorem2} and \ref{maintheorem3} for $G_2$.

\begin{theorem} \label{theorem:G2}
Suppose $p> 3$, $\Phi$ is of type $G_{2}$, and $\lambda\leq \omega_{j}$. Then 
\begin{enumalph}
\item $\res:\opH^{1}(G,L(\lambda))\rightarrow \opH^{1}(\Gfq,L(\lambda))$
is an isomorphism;
\item $\opH^{1}(G,L(\lambda))=\opH^{1}(\Gfq,L(\lambda))=0$ 
\end{enumalph} 
\end{theorem}

\begin{proof} First note that part (b) implies part (a), so we will prove part (b). When $\lambda\leq \omega_{j}$, we have 
$L(\lambda)=H^{0}(\lambda)$  \cite[\S 4.6]{Jantzen:1991}. Therefore, $\opH^{1}(G,L(\lambda))=0$ \cite[II.4.13]{Jantzen:2003}. From Theorem~\ref{theorem:relatecoho}, we have an injective map from $\opH^{1}(G(\Fq),L(\lambda))\hookrightarrow \opH^{1}(U_{1},L(\lambda))^{\Tfq}$.  According to Theorem~\ref{theorem:U1-coho} and Corollary \ref{corollary:corsoc}, we have for $p>5$, 
\[
\opH^{1}(U_{1},L(\lambda))\cong \Ext^1_{U_1}(L(-w_{0}\lambda), k) = \bigoplus_{\alpha \in \Delta} -s_\alpha \cdot (-w_{0} \lambda ) = \bigoplus_{\alpha \in \Delta} -s_\alpha \cdot \lambda.
\]
We observe that the proof of Theorem \ref{theorem:U1-coho} also works in the case when $p>3$ and $\lambda \in \set{0,\omega_1}$. Now for $p>3$, $\alpha \in \Delta$, and $\lambda \in \set{0,\omega_1,\omega_2}$, the weight $-s_\alpha \cdot \lambda$ is not divisible by $q-1$. Thus, for $p > 3$ and $\lambda \in \set{0,\omega_1}$, and for $p > 5$ and $\lambda = \omega_2$, we have $\opH^{1}(G(\Fq),L(\lambda)) = \opH^{1}(U_{1},L(\lambda))^{\Tfq}=0$. 

We now consider the case $p=5$ and $\lambda=\omega_{2}$. By \cite[I.9.19]{Jantzen:2003}, there exists an injective map from $\opH^{1}(U_{1},L(\omega_{2}))$ into the ordinary Lie algebra cohomology group $\opH^{1}(\fraku,L(\omega_{2}))$. Taking $\Tfq$-fixed points yields an injection $\opH^{1}(U_{1},L(\omega_{2}))^{\Tfq}\hookrightarrow \opH^{1}(\fraku,L(\omega_{2}))^{\Tfq}$. Next observe that $\opH^{1}(\fraku,L(\omega_{2}))$ is a $T$-subquotient of $\opH^{1}(\fraku,k)\otimes L(\omega_{2})\cong (\fraku/[\fraku,\fraku])^{*}\otimes L(\omega_{2})$ (cf.\ the proof of \cite[Proposition 2.5.1]{UGA:2009}). The weights of $(\fraku/[\fraku,\fraku])^{*}$ are $\alpha_{1}$ and $\alpha_{2}$, and the module $L(\omega_{2})$ is the adjoint representation. Then when $p=5$, one can show that the only $\Tfq$-invariants in $(\fraku/[\fraku,\fraku])^{*}\otimes L(\omega_{2})$ are the $T$-invariants, and hence that $\opH^{1}(\fraku,L(\omega_{2}))^{\Tfq}=\opH^{1}(\fraku,L(\omega_{2}))^{T}$. 

Suppose $\opH^{1}(\fraku,L(\omega_{2}))^{T}\neq 0$. Then by \cite[Theorem 2.4.1]{UGA:2009}, there exist $w\in W$ and $\nu \in X(T)$ such that $-w\cdot(-w_{0}\omega_{2})+5\nu=-w\cdot \omega_{2}+5\nu=0$. Then $(w\cdot \omega_2, \beta^{\vee})\in 5 \Z $ for all $\beta\in \Phi$. Observe that $(w\cdot \omega_{2},\beta^{\vee})=(\omega_{2}+\rho,w^{-1}\beta^{\vee})-(\rho,\beta^{\vee})\in 5 \Z$. Set $\beta=\alpha_{0}$. Then $\beta^{\vee}=2\alpha_{1}^{\vee}+3\alpha_{2}^{\vee}$. This implies that $(\omega_{2}+\rho,w^{-1}\alpha_{0}^{\vee})\in 5 \Z$. But this is a contradiction, because the possibilities for $w^{-1}\alpha_{0}^{\vee}$ are $\{\pm (2\alpha_{1}^{\vee}+3\alpha_{2}^{\vee}), \pm \alpha_{1}^{\vee}, \pm (\alpha_{1}^{\vee}+3\alpha_{2}^{\vee})\}$, and none of these have inner product with $\omega_{2}+\rho$ that is divisible by 5. So $\opH^{1}(\fraku,L(\omega_{2}))^{T}=0$, and hence $\opH^{1}(\Gfq,L(\omega_{2}))=0$.
\end{proof}

We now prove Theorem \ref{maintheorem1} for the remaining Lie types.

\begin{proof}[Proof of Theorem \ref{maintheorem1}]
The case when $\Phi$ is of type $G_2$ is handled by Theorem \ref{theorem:G2}, so we assume for the remainder of the proof that $\Phi$ is not of type $G_2$. By Remark \ref{remark:lambda0} we may also assume $\lambda\ne 0$. By Theorem \ref{theorem:U1-coho},
\[
\opH^1(U_1,L(\lambda)) \cong \Ext^1_{U_1}(L(\lambda^*), k) = \bigoplus_{\alpha \in \Delta} -s_\alpha \cdot \lambda^* 
\oplus \bigoplus_{\sigma\uparrow \lambda^*} (-\sigma)^{\oplus m_\sigma},
\]
where $m_\sigma=\dim \Ext^1_G(L(\lambda^*),H^0(\sigma))$. Note that $\lambda^* = -w_0\lambda$ is again dominant and less than or equal to a fundamental dominant weight. Consider $-s_\alpha \cdot \lambda^*=-\lambda^*+(\lambda^*+\rho,\alpha^\vee)\alpha$. Consulting the lists, provided in Section \ref{subsection:Hassediagrams}, of dominant weights less than or equal to a fundamental dominant weight, and using the Cartan matrix to rewrite $\alpha$ as a sum of fundamental dominant weights, one can check that the coefficients of $-s_{\alpha}\cdot \lambda^*$ are not all divisible by $q-1$ when $q>3$, and hence that the weight $-s_\alpha \cdot \lambda^*$ does not contribute to the $\Tfq$-invariants in $\opH^1(U_1,L(\lambda))$. Consequently, 
\[
\opH^1(U_1,L(\lambda))^{\Tfq}= \bigoplus_{\sigma\uparrow \lambda^*} [(-\sigma)^{\oplus m_\sigma}]^{\Tfq}.
\]
Now observe that since $\lambda^*$ is less than or equal to a fundamental dominant weight, the only weight $\sigma \in X(T)_+$ satisfying $\sigma < \lambda^*$ that gives a $\Tfq$-invariant in $\opH^1(U_1,L(\lambda))$ is the zero weight. Then $\dim \opH^{1}(U_{1},L(\lambda))^{\Tfq} = m_0 = \dim \opH^1(G,L(\lambda))$. 

Consider the commutative diagram \eqref{equation:keycommutativediagram}. The top restriction map factors as $\opH^{1}(G,L(\lambda)) \to \opH^{1}(B,L(\lambda)) \to \opH^{1}(U,L(\lambda))^{T}$, and both of these maps are well-known isomorphisms \cite[I.6.9, II.4.7]{Jantzen:2003}. Composing the second of these with the right vertical restriction map in \eqref{equation:keycommutativediagram} gives the restriction $\opH^{1}(B,L(\lambda)) \to \opH^{1}(U_{1},L(\lambda))^{\Tfq}$, which (by dualizing Corollary \ref{corollary:BtoU1injection}) is an injection. Therefore, the composite restriction $\opH^{1}(G,L(\lambda)) \to \opH^{1}(U,L(\lambda))^{T} \to \opH^{1}(U_{1},L(\lambda))^{\Tfq}$ is injective. But the first and last spaces have the same dimension, $m_{0}$, so the composition is an isomorphism. Therefore the bottom map $\Gamma$ in \eqref{equation:keycommutativediagram} is surjective. But by Theorem \ref{theorem:relatecoho}, $\Gamma$ is injective, so it is an isomorphism. Finally, we conclude that the left vertical restriction map $\opH^{1}(G,L(\lambda)) \to \opH^{1}(\Gfq,L(\lambda))$ must also be an isomorphism.
\end{proof}

\begin{remark}
We note that there are examples for which $\opH^1(U_1,L(\omega_j))^{\Tfp} \neq 0$ but for which $\opH^1(\Gfp,L(\omega_j)) = 0$. For example, suppose $\Phi$ has type $A_n$, $n \geq 2$, $j=2$, and $p=3$. Then $\Gfp = SL_{n+1}(\F_3)$. For $i\neq j$, $-s_{\alpha_i} \cdot \omega_j = -\omega_j +\alpha_i$. Also, $\alpha_1 = 2\omega_1 - \omega_2$. Then $-s_{\alpha_1} \cdot \omega_2 = 2 \omega_1 - 2 \omega_2 = (p-1)(\omega_1 - \omega_2)$, which yields a $T(\F_3)$-invariant in $\opH^1(U_1,L(\omega_2))$. On the other hand, $\opH^1(G(\F_3), L(\omega_2))=0$ by \cite[Proposition 8.5]{Jones:1975}. So the condition $q>3$ is essential in type $A_{n}$.
\end{remark}

\subsection{Vanishing conditions} Theorem \ref{maintheorem1} lets us reduce the problem of computing the cohomology group $\opH^1(\Gfq,L(\lambda))$ for the finite group $\Gfq$ to the problem of computing the corresponding cohomology group for the full algebraic group $G$, where results are typically easier to obtain. In particular, we can apply the Linkage Principle for $G$ as well as standard facts on induced and Weyl modules for $G$ to deduce conditions under which the cohomology group $\opH^1(\Gfq,L(\lambda))$ vanishes.

\begin{theorem} \label{theorem:finitegroupvanishing}
Assume that $p > 2$ when $\Phi$ is of type $A_{n}$ or $D_{n}$, $p>5$ when $\Phi$ is of type $E_8$, and $p > 3$ in all other cases. Assume that $q > 3$. Let $\lambda \in X(T)_+$ with $\lambda \leq \omega_j$ for some fundamental dominant weight $\omega_j$. 
Then $\opH^1(\Gfq,L(\lambda)) = 0$ if either $L(\lambda) = H^0(\lambda)$ or if $\lambda$ is not linked to zero under the dot action of the affine Weyl group $W_p$.
\end{theorem}

\begin{proof}
By Theorem \ref{maintheorem1}, the restriction map $\opH^1(G,L(\lambda)) \rightarrow \opH^1(\Gfq,L(\lambda))$ is an isomorphism. If $L(\lambda) = H^0(\lambda)$, then $\opH^i(G,L(\lambda)) = \Ext_G^i(V(0),H^0(\lambda)) = 0$ for all $i \geq 1$ by \cite[II.4.13]{Jantzen:2003}. If $\lambda$ is not linked to zero under the dot action of the affine Weyl group $W_p$, then neither is $\lambda^*$, and $\opH^i(G,L(\lambda)) \cong \Ext_G^i(L(\lambda^*),H^0(0)) = 0$ for all $i \geq 0$ by \cite[II.6.20]{Jantzen:2003}.
\end{proof}

\subsection{1975 CPS result}

Recall that a non-zero weight $\lambda$ is called \emph{minuscule} if for all $\alpha \in \Phi$, $(\lambda,\alpha^\vee) \in \set{-1,0,1}$. A list of the dominant minuscule weights for each indecomposable root system is given in  Table \ref{table:minusculeweights}. Note that the minuscule weights are all fundamental dominant weights, and are not elements of the root lattice, as can be seen by comparing the list in Table \ref{table:minusculeweights} with the data in \cite[Table 13.1]{Humphreys:1978}.

\begin{table}[htbp]
\begin{tabular}{ll}
\hline
Type & Minuscule Weights \\
\hline
$A_n$ & $\omega_i$, $(1 \leq i \leq n)$ \\
$B_n$ & $\omega_n$ \\
$C_n$ & $\omega_1$ \\
$D_n$ & $\omega_1,\omega_{n-1},\omega_n$ \\
$E_6$ & $\omega_1,\omega_6$ \\
$E_7$ & $\omega_7$ \\
$E_8$ & none \\
$F_4$ & none \\
$G_2$ & none \\
\hline
& \\
\end{tabular}
\caption{List of minuscule weights.} \label{table:minusculeweights}
\end{table}

We now recover results of Cline, Parshall and Scott \cite{Cline:1975} for minuscule highest weights.

\begin{corollary} \label{corollary:minusculevanishing}
Assume that $p>2$ when $\Phi$ is of type $A_n$ or $D_n$, $p>5$ when $\Phi$ is of type $E_8$, and $p > 3$ in all other cases. Assume also that $q > 3$. Let $\omega_j$ be a minuscule dominant weight. Then
\[
\opH^1(\Gfq, L(\omega_j)) = 0.
\]
\end{corollary}

\begin{proof}
As $\omega_j$ is not an element of the root lattice $\Z\Phi$, it cannot be linked to zero under the dot action of the affine Weyl group $W_p$. Now apply Theorem \ref{theorem:finitegroupvanishing}.
\end{proof}

\section{Results for Classical Groups} \label{section:classicalgroups}

\subsection{Types \texorpdfstring{$A_n$, $B_n$, $D_n$}{An, Bn, Dn}} \label{section:ABD}

For the classical groups, the condition $\lambda \leq \omega_{j}$ implies that $\lambda=0$ or that $\lambda = \omega_i$ for some $1 \leq i \leq j$. We proceed to verify Theorems \ref{maintheorem2} and \ref{maintheorem3} for the classical groups. For types $A_{n}$, $B_{n}$, and $D_{n}$, if $p>2$ then  $L(\lambda) = H^0(\lambda) = V(\lambda)$ by \cite[II.8.21]{Jantzen:2003}. Then under the hypotheses of Theorems \ref{maintheorem2} and \ref{maintheorem3}, $\opH^1(\Gfq,L(\lambda)) = 0$ by Theorem \ref{theorem:finitegroupvanishing}.

\subsection{Non-vanishing in type \texorpdfstring{$C_n$}{Cn}}

Assume that $\Phi$ has type $C_n$ with $n \geq 3$. As noted in the Appendix, $\lambda \leq \omega_i$ implies that $\lambda = \omega_j$ for some $0 \leq j\leq i$ (where $\omega_0:=0$). Kleshchev and Sheth \cite{Kleshchev:1999,Kleshchev:2001}, using results of Adamovitch on the submodule structure of the Weyl modules $V(\omega_j)$, completely determine the structure $\opH^1(G, L(\omega_j))$. We formulate their result as follows.

\begin{theorem}[Kleshchev--Sheth] \label{theorem:KlSh}
Let $\Phi$ be of type $C_n$ with $n \geq 3$ and $p>2$. Write $n+1 = b_0 + b_1 p+ \cdots + b_t p^t$ with $0 \leq b_i < p$ and $b_t \neq 0$. Then
\[
\opH^{1}(G,L(\omega_j)) \cong
\begin{cases} k & \text{if $j=2b_{i}p^{i}$ for some $0 \leq i < t$ with $b_{i} \neq 0$}, \\ 0 & \text{otherwise}.
\end{cases}
\]
\end{theorem}

\begin{proof}
By \cite[Corollary 3.6(ii)]{Kleshchev:2001}, $\opH^{1}(G,L(\omega_{j}))$ is isomorphic to either $0$ or $k$, and is isomorphic to $k$ under precisely the following conditions: $n+1-j = a_{0}+a_{1}p + \dots + a_{s}p^{s}$ where $0\le a_{i}<p$, and $j=2(p-a_{i})p^{i}$ for some $i$ such that $a_{i}>0$ and either $a_{i+1}<p-1$ or $j<2p^{i+1}$. But if $j=2(p-a_{i})p^{i}$ and $a_{i}>0$, then $j=2p^{i+1}-2a_{i}p^{i} < 2p^{i+1}$, so the ``either\dots or\dots'' condition is always true and thus superfluous. Writing $j=(p-2a_{i})p^{i}+p^{i+1}$, we have
\begin{equation*}
\setlength{\arraycolsep}{2pt}
\renewcommand{\arraystretch}{1.3}
\begin{array}{rlrrrl}
n+1 =& a_{0}+\dots &+&a_{i}p^{i}\ +& a_{i+1}p^{i+1}&+\dots + a_{s}p^{s}\\
          &              &+&(p-2a_{i})p^{i}\ +&               p^{i+1}&\\
        =& a_{0}+\dots &+&(p-a_{i})p^{i}\ +& (a_{i+1}+1)p^{i+1}&+\dots + a_{s}p^{s}.
\end{array}
\end{equation*}
This has the form given in the statement of the theorem, with $0<b_{i}=p-a_{i}<p$, $j=2b_{i}p^{i}$, and at least one nonzero term beyond $b_{i}p^{i}$ in the $p$-adic expansion of $n+1$. 
\end{proof}

Combining Theorem~\ref{theorem:KlSh} with Theorem~\ref{maintheorem1} we obtain:

\begin{corollary} \label{corollary:KlSh}
Let $\Phi$ be of type $C_n$ with $n \geq 3$ and $p>3$. Write $n+1 = b_{0}+b_{1}p+\dots + b_{t}p^{t}$ with $0\le b_{i}<p$ and $b_{t}\ne 0$. Let $\lambda \in X(T)_+$  with $\lambda \leq \omega_j$ for some $j$. Then
\[
\opH^{1}(\Gfq, L(\lambda)) \cong
\begin{cases} k & \text{if $\lambda=\omega_{j}$, $j=2b_{i}p^{i}$ for some $0\le i < t$ with $b_{i}\ne 0$,} \\ 0 & \text{otherwise}.
\end{cases}
\]
\end{corollary}

The vanishing of $\opH^1(\Gfq,L(\omega_j))$ for $j$ odd and $p>3$ can also be seen from Theorem \ref{theorem:finitegroupvanishing}, since for odd $j$ the fundamental weight $\omega_j$ is not in the root lattice, hence is not linked to zero under the dot action of the affine Weyl group $W_p$.

\section{Results for Exceptional Groups} \label{section:exceptionalgroups}

\subsection{Large prime vanishing}

The results of Section \ref{section:applications} completely compute $\opH^1(\Gfq,L(\lambda))$ when the underlying root system $\Phi$ is of classical type, when $p > 3$ (or $p>2$ and $q>3$ for types $A_n$ and $D_n$), and when $\lambda$ is less than or equal to a fundamental dominant weight. For the exceptional types a number of open cases remain. The following lemma narrows down the list of remaining open cases to only finitely many values of $p$.

\begin{lemma} \label{lemma:largeprimevanishing}
Let $\lambda \in X(T)_+$ with $\lambda \leq \omega_j$ for some fundamental dominant weight $\omega_j$. Set $h_\lambda = (\lambda,\alpha_0^\vee)$, and suppose $p \geq \max \{ h+h_\lambda-1,h_\lambda + 4 \}$. Then $\opH^1(\Gfq,L(\lambda)) = 0$.
\end{lemma}

\begin{proof}
By Corollary \ref{corollary:vanishingresult}, it suffices to show that $\opH^1(U_1,L(\lambda))^{\Tfq} = 0$. We will show that the possibly larger space $\opH^1(U_1,L(\lambda))^{\Tfp}$ is zero.

By Lemma \ref{lemma:H1weights}, the weights of $\opH^1(U_1,L(\lambda))$ have the form $\beta + \nu$ for some $\beta \in \Delta$ and some weight $\nu$ of $L(\lambda)$. The weight $\beta + \nu$ contributes to the $\Tfp$-invariants of $\opH^1(U_1,L(\lambda))$ only if $\beta + \nu = (p-1)\mu$ for some $\mu \in X(T)$. So suppose $\beta + \nu = (p-1)\mu$. Choose $y \in W$ such that $y\mu \in X(T)_+$. Then $y\beta \in \Phi$ and $y\nu$ is a weight of $L(\lambda)$. If $\mu \neq 0$, then 
$$p-1 \leq (p-1)(y\mu,\alpha_0^\vee) = (y\beta,\alpha_0^\vee) + (y\nu,\alpha_0^\vee) \leq 2 + (\lambda,\alpha_0^\vee) = 2+h_\lambda$$ 
if $\Phi$ does not have type $G_2$, and $p-1 \leq 3+h_\lambda$ if $\Phi$ does have type $G_2$. In either case, the inequality contradicts the assumption $p \geq \max\{ h + h_\lambda-1, h_\lambda+4 \}$ (recall that $h=6$ in type $G_{2}$), so we must have $\mu = 0$. This implies that $\opH^1(U_1,L(\lambda))^{\Tfp} = \opH^1(U_1,L(\lambda))^T$.

Set $\fraku = \Lie(U)$. By \cite[Proposition 1.1]{Friedlander:1986}, there exists a spectral sequence of $B$-modules satisfying
\[
E_2^{2i,j} = S^i(\fraku^*)^{(1)} \otimes \opH^j(\fraku,L(\lambda)) \Rightarrow \opH^{2i+j}(U_1,L(\lambda)),
\]
with $E_2^{i,j}=0$ if $i$ is odd. Then $\opH^1(U_1,L(\lambda))$ is a $T$-module subquotient of $\opH^1(\fraku,L(\lambda))$. We claim that $\opH^1(\fraku,L(\lambda))^T = 0$. The assumption $p \geq h+h_\lambda - 1$ implies for all $\beta \in \Phi^+$ that $(\lambda+\rho,\beta^\vee) \leq p$. Then by \cite[Theorem 4.2.1]{UGA:2009}, the weights of $T$ in $\opH^1(\fraku,L(\lambda))$ are precisely 
$\{ -s_\alpha \cdot \lambda^*: \alpha \in \Delta \}$.\footnote{The introduction to the paper \cite{UGA:2009} states that the Borel subgroup $B$ and its unipotent radical $U$ should correspond to the set of negative roots in $\Phi$. However, for the theorems to be correctly stated 
these groups and the Lie algebra $\mathfrak u = \Lie(U)$ should correspond the set of positive roots in $\Phi$. The weights we have listed here for $\opH^1(\mathfrak u,L(\lambda))$ are the correct weights when the Lie algebra $\mathfrak u$ corresponds to the set of negative roots in $\Phi$.} 
The weight $\lambda^* = -w_0\lambda$ is again a dominant weight less than or equal to a fundamental dominant weight. Now one checks for all $\alpha \in \Delta$ that $-s_\alpha \cdot \lambda^* \neq 0$, and hence $\opH^1(\mathfrak u,L(\lambda))^T=0$. 
So then also $\opH^1(U_1,L(\lambda))^{\Tfp} = \opH^1(U_1,L(\lambda))^T = 0$.
\end{proof}

\subsection{Small prime vanishing}

We now handle many of the smaller values for $p$ not covered by Lemma \ref{lemma:largeprimevanishing}.

\begin{proposition} \label{proposition:genericvanishing}
Let $\lambda\in X(T)_+$  with $\lambda \leq \omega_j$ for some $j$. Suppose that $3 < p \leq 31$, $\Phi$ is of exceptional type, and $p$ does not equal one of the primes listed next to $\lambda$ in the Hasse diagram for $\Phi$ appearing in Section \ref{subsection:Hassediagrams}. If $\Phi = E_8$, assume also that $p>5$. Then $\opH^1(\Gfq, L(\lambda)) = 0$. The conclusion $\opH^1(\Gfq, L(\lambda)) = 0$ also holds if $p=7$, $\Phi = E_8$, and $\lambda = \omega_3$.
\end{proposition} 

\begin{proof}
If $3< p \leq 31$ and $p$ does not equal one of the primes listed next to $\lambda$ in the Hasse diagram for $\Phi$, then $\lambda$ is not linked to zero under the dot action of the affine Weyl group $W_p$. If $p=7$, $\Phi = E_8$, and $\lambda = \omega_3$, then $L(\lambda) = H^0(\lambda)$ by \cite[\S 4.6]{Jantzen:1991}. In any case, $\opH^1(\Gfq,L(\lambda)) = 0$ by Theorem \ref{theorem:finitegroupvanishing}.
\end{proof}

When $\Phi$ is of exceptional type, the largest possible value for $h+h_{\lambda}-1$ in Lemma \ref{lemma:largeprimevanishing} is 35, which occurs for type $E_{8}$ when $3\omega_{8} \le \lambda \le \omega_{4}$. Thus Lemma \ref{lemma:largeprimevanishing} and Proposition \ref{proposition:genericvanishing} show that the only cases of Theorems \ref{maintheorem2} and \ref{maintheorem3} which we have not thus far explicitly calculated are $\opH^1(\Gfq,L(\lambda))$ when the underlying root system is of exceptional type and when the prime $p$ is one of those appearing next to the weight $\lambda$ in the Hasse diagram for $\Phi$ in Section \ref{subsection:Hassediagrams}.

\subsection{Non-zero cohomology groups for types \texorpdfstring{$\mathbf{F_4}$, $\mathbf{E_7}$, and $\mathbf{E_8}$}{F4, E7, and E8}}

Let $\E$ be the Euclidean space spanned by $\Phi$. Recall that the affine Weyl group $W_p$ is generated by the set of simple reflections $\set{s_\alpha : \alpha \in \Delta} \subset GL(\E)$ together with the affine reflection $s_0 := s_{\alpha_0,p}$, which is defined for $\lambda \in \E$ by
\[ s_0(\lambda) = \lambda - ((\lambda,\alpha_0^\vee)-p)\alpha_0 = s_{\alpha_0}(\lambda)+p\alpha_0.
\]

\begin{theorem} \label{theorem:nonvanishingtranslation}
Suppose that one of the following conditions is satisfied:
\begin{enumerate}
\item $\Phi$ has type $F_4$ and $p=13$, so $s_0 \cdot 0 = 2\omega_4$;
\item $\Phi$ has type $E_7$ and $p=19$, so $s_0 \cdot 0 = 2\omega_1$; or
\item $\Phi$ has type $E_8$ and $p=31$, so $s_0 \cdot 0 = 2\omega_8$.
\end{enumerate}
Then $\opH^1(\Gfq,L(s_0 \cdot 0)) \cong k$.
\end{theorem}

\begin{proof}
For the root systems under consideration, the involution $\mu \mapsto \mu^* := -w_0\mu$ on $X(T)$ is the identity, so by Theorem \ref{maintheorem1}, $\opH^1(\Gfq,L(s_0 \cdot 0)) \cong \opH^1(G,L(s_0 \cdot 0)) \cong \Ext_G^1(L(s_0 \cdot 0),k)$. Now
\[
\Ext_G^i(L(s_0 \cdot 0),k) \cong \begin{cases} k & \text{if $i=1$, and} \\ 0 & \text{if $i \neq 1$} \end{cases}
\]
by the case $w=1$ of \cite[II.7.19(c)]{Jantzen:2003}.
\end{proof}

The following additional non-vanishing result holds for type $E_7$.

\begin{theorem}
Suppose $\Phi$ has type $E_7$ and $p=7$. Then $\opH^1(\Gfq,L(\omega_6)) \cong k$.
\end{theorem}

\begin{proof}
First, $\opH^1(\Gfq,L(\omega_6)) \cong \opH^1(G,L(\omega_6))$ by Theorem \ref{maintheorem1}. Next, if $p=7$ then $\dim H^0(\omega_6) = 1+\dim L(\omega_6)$ by \cite[Table, p.\ 414]{Gilkey:1988}. Equivalently, $\dim V(\omega_6) = 1+\dim L(\omega_6)$. The Weyl module $V(\omega_6)$ has head isomorphic to $L(\omega_6)$, so we conclude for $p=7$ that $\rad_G V(\omega_6) \cong k$. Then $\opH^1(G,L(\omega_6)) \cong \Ext_G^1(L(\omega_6),k) \cong k$ by \cite[II.2.14]{Jantzen:2003}.
\end{proof}

\subsection{Vanishing via translation functors for type \texorpdfstring{$\mathbf{E_8}$}{E8}}

Suppose $\Phi$ has type $E_8$ and $p=31$. Let $\lambda \in X(T)_+$, and suppose $\opH^1(G,L(\lambda)) \cong \Ext_G^1(L(\lambda),k) \neq 0$. Then by \cite[II.2.14]{Jantzen:2003}, the trivial module $L(0) \cong k$ must appear as a composition factor of $H^0(\lambda)$. We will show that this cannot happen if $\lambda \in \set{\omega_7+\omega_8,\omega_6+\omega_8}$.

First suppose $\lambda = \omega_7+\omega_8$. Consider the translation functor $T_0^{\omega_8}$. We have $T_0^{\omega_8}(L(0)) \cong L(\omega_8)$ by \cite[II.7.15]{Jantzen:2003}. Also, $(s_0s_8) \cdot 0 = \omega_7+\omega_8$ and $(s_0s_8) \cdot \omega_8 = 2\omega_7 - \omega_8 \notin X(T)_+$, so $T_0^{\omega_8}(H^0(\omega_7+\omega_8)) \cong H^0(2\omega_7-\omega_8) = 0$ by \cite[II.7.11]{Jantzen:2003}. But $T_0^{\omega_8}$ is an exact functor, so if $L(0)$ occurred as a composition factor in $H^0(\omega_7+\omega_8)$, we would have $T_0^{\omega_8}(L(0)) = 0$, a contradiction. Similarly, $(s_0s_8s_7) \cdot 0 = \omega_6 + \omega_8$ and $(s_0 s_8 s_7) \cdot \omega_8 = \omega_6+\omega_7-\omega_8 \notin X(T)_+$, so $T_0^{\omega_8}(H^0(\omega_6+\omega_8)) = 0$, and $L(0)$ cannot occur as a composition factor in $H^0(\omega_6+\omega_8)$. We have proved:

\begin{theorem} \label{theorem:vanishingviatranslation}
Suppose $\Phi$ has type $E_8$ and $p=31$. Let $\lambda \in \set{\omega_7+\omega_8,\omega_6+\omega_8}$. Then
\[ \opH^1(\Gfq,L(\lambda)) \cong \opH^1(G,L(\lambda)) = 0. \]
\end{theorem}

In summary, for $p > 3$ (and $p>5$ when $\Phi=E_8$), and for $\lambda$ less than or equal to a fundamental dominant weight, we have computed all cohomology groups $\opH^1(\Gfq,L(\lambda))$ except for 2 cases  in type $E_7$, $\lambda=2\omega_{7},\ p=5$, and $\lambda=\omega_{2}+\omega_{7},\ p=7$ indicated in Figure \ref{figure:HasseE7}, and the 3 cases $\lambda \in \set{2\omega_7,\omega_1+\omega_7,\omega_2+\omega_8}$ for $p=7$ in type $E_8$ indicated in Figure \ref{figure:HasseE8}. In all the cases we have computed, we have found $\dim \opH^1(\Gfq,L(\lambda)) \le 1$.

\section{Appendix}

\subsection{Hasse diagrams for fundamental weights} \label{subsection:Hassediagrams}

In this section we describe the restriction of the partial ordering $\leq$ on $X(T)$ to the dominant weights $\lambda \in X(T)_+$ satisfying $\lambda \leq \omega_j$ for some fundamental dominant weight $\omega_j$. If $\Phi$ has classical type $A_n,B_n,C_n$ or $D_n$, then all such $\lambda$ are themselves fundamental dominant weights (or 0).

\bigskip

\noindent \textbf{Type $\mathbf{A_n}$.} Each fundamental dominant weight $\omega_j$ is minimal with respect to $\leq$.

\bigskip

\noindent \textbf{Type $\mathbf{B_n}$.} The weight $\omega_n$ is minimal with respect to $\leq$. The remaining fundamental dominant weights satisfy $\omega_{n-1} > \omega_{n-2} > \cdots > \omega_2 > \omega_1 > 0$.

\bigskip

\noindent \textbf{Type $\mathbf{C_n}$.} Set $\omega_0 = 0$. Then the fundamental dominant weights form two independent chains, $\omega_n > \omega_{n-2} > \omega_{n-4} > \cdots$ and $\omega_{n-1} > \omega_{n-3} > \omega_{n-5} > \cdots$, with $\omega_0$ appearing as the smallest term in the chain having even indices, and $\omega_1$ appearing as the smallest term in the chain with odd indices.

\bigskip

\noindent \textbf{Type $\mathbf{D_n}$.} The weights $\omega_n$, $\omega_{n-1}$, and $\omega_1$ are minimal with respect to $\leq$. The remaining fundamental dominant weights satisfy $\omega_{n-2} > \omega_{n-3} > \cdots > \omega_1$.

\bigskip

\noindent \textbf{Exceptional types.} The Hasse diagrams for the exceptional types are given in Figures \ref{figure:HasseE6}--\ref{figure:HasseG2}. If the weight $\mu$ appears below and is connected to the weight $\lambda$ by a line, then $\mu < \lambda$. White boxes with black borders indicate that the given weight is conjugate (i.e., linked) to $0$ under the dot action of the affine Weyl group $W_p$ for the primes shown. Linkage relations were verified for $5 \leq p \leq 31$ using the computer program GAP \cite{GAP4}.

  \pgfdeclarelayer{background} 
  \pgfdeclarelayer{foreground} 
  \pgfsetlayers{background,main,foreground}

  \tikzstyle{wt}=[dashed,fill=black!0,text=black,rounded corners]
  \tikzstyle{prime} = [scale=.8, text=black, draw=black!50, rounded corners, draw, thin, fill=white]
  \tikzstyle{linkw0}=[rounded corners,draw=black!50,thin,fill=orange!0,text=black]
  
  \tikzset{every picture/.style={line width=.8pt}}
% \tikzstyle{linkw0}=[rounded corners,draw,fill=orange!20,text=black]

\begin{figure}[htbp]
\begin{tikzpicture}[scale=1]%
% \tikzstyle{wt}=[fill=violet!17!magenta!15,text=black!85,rounded corners]
% \node[fill=white,scale=1.5] at (2,1) {$E_6$};
  \draw (0,0) node[wt] (om5) {$\omega_5$}
  -- ++ (0,-1) node[wt] (om1) {$\omega_1$};

  \draw (2,0) node[wt] (om4) {$\omega_4$}
  -- ++ (0,-1) node[wt] (om1om6) {$\omega_1 + \omega_6$}
  -- ++ (0,-1) node[wt] (om2) {$\omega_2$}
  -- ++ (0,-1) node[linkw0] (om0) {$0$};

  \draw (4,0) node[wt] (om3) {$\omega_3$}
  -- ++ (0,-1) node[wt] (om6) {$\omega_6$};
\end{tikzpicture}
\caption{Hasse diagram for $E_6$.} \label{figure:HasseE6}
\end{figure}
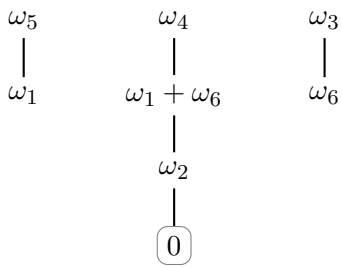

\begin{figure}[htbp]
\begin{tikzpicture}[scale=1]%
% \tikzstyle{wt}=[fill=green!20,text=black!85,rounded corners]
% \node[fill=white,scale=1.5] at (2,1) {$E_7$};
  \draw (0,0) node[wt] (om4) {$\omega_4$}
  -- ++ (0,-1) node[wt] (om1om6) {$\omega_1 + \omega_6$}
  ++ (1,-1) node[linkw0] (2om1) {$2 \omega_1$} +(.45,.2) node[prime] {19}
  ++ (-2,0) node[linkw0] (om2om7) {$\omega_2+\omega_7$} +(.75,.25) node[prime] {7}
  + (1,-1) node[wt] (om3) {$\omega_3$}
  ++ (-1,-1) node[linkw0] (2om7) {$2 \omega_7$} +(.38,.2) node[prime] {5}
  ++ (1,-1) node[linkw0] (om6) {$\omega_6$} +(.38,.2) node[prime] {7}
  ++ (0,-1) node[wt] (om1) {$\omega_1$}
  -- ++ (0,-1) node[linkw0] (om0) {$0$};

  \begin{pgfonlayer}{background}
    \pgfsetlinewidth{0.8pt}
  \draw (2om1) -- (om1om6) -- (om2om7) -- (2om7) -- (om6) -- (om1);
  \draw (2om1) -- (om3) -- (om2om7) -- (om3) -- (om6);

  \draw (3,-1) node[wt] (om5) {$\omega_5$}
  -- ++ (0,-1) node[wt] (om1om7) {$\omega_1 + \omega_7$}
  -- ++ (0,-1) node[wt] (om2) {$\omega_2$}
  -- ++ (0,-1) node[wt] (om7) {$\omega_7$};
  \end{pgfonlayer}
\end{tikzpicture}
\caption{Hasse diagram for $E_7$.} \label{figure:HasseE7}
\end{figure}
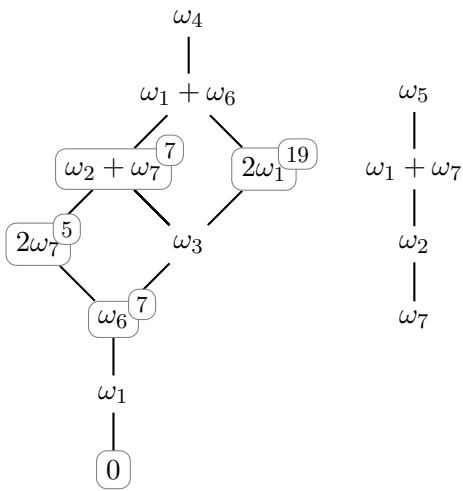

\begin{figure}[htbp]
\begin{tikzpicture}[scale=1]%
% \tikzstyle{wt}=[fill=blue!20,text=black!85,rounded corners]
% \node[fill=white,scale=1.5] at (0,1) {$E_8$};
 % draw the nodes
  %\begin{pgfonlayer}{foreground}
  \draw (0,0) node[linkw0] (om4) {$\omega_4$} +(.35,.25) node[prime] {5}
  ++ (0,-1) node[wt] (om1om6) {$\omega_1 + \omega_6$}
  ++ (-1,-1) node[wt] (2om1om8) {$2\omega_1 + \omega_8$}
  ++ (2,0) node[wt] (om2om7) {$\omega_2 + \omega_7$}
  ++ (1,-1) node[linkw0] (2om7) {$2 \omega_7$} +(.38,.25) node[prime] {7}
  ++ (-2,0) node[wt] (om3om8) {$\omega_3 + \omega_8$}
  ++ (-1,-1) node[linkw0] (om1om2) {$\omega_1 + \omega_2$} +(.75,.25) node[prime] {5}
  ++ (2,0) node[linkw0] (om6om8) {$\omega_6 + \omega_8$} +(.75,.25) node[prime] {31}
  ++ (1,-1) node[wt] (om12om8) {$\omega_1 + 2 \omega_8$}
    + (2.3,-1.3) node[wt] (3om8) {$3 \omega_8$}
  ++ (-2,0) node[linkw0] (om5) {$\omega_5$} +(.38,.2) node[prime] {5}
  ++ (1,-1) node[linkw0] (om1om7) {$\omega_1 + \omega_7$} +(.75,.25)
  node[prime] {7}
  ++ (-1,-1) node[linkw0] (2om1) {$2 \omega_1$} +(.38,.2) node[prime] {5}
  ++ (2,0) node[linkw0] (om2om8) {$\omega_2 + \omega_8$} +(.8,.3)
  node[prime] {5, 7}
  ++ (1,-1) node[linkw0] (om7om8) {$\omega_7 + \omega_8$} +(.75,.25) node[prime] {31}
  ++ (-2,0) node[linkw0] (om3) {$\omega_3$} +(.38,.2) node[prime] {7}
  ++ (1,-1) node[linkw0] (om6) {$\omega_6$} +(.38,.2) node[prime] {5}
  ++ (0,-1) node[linkw0] (om1om8) {$\omega_1 + \omega_8$} +(.75,.25) node[prime] {5}
  ++ (-1,-1) node[linkw0] (om2) {$\omega_2$} +(.38,.2) node[prime] {5}
  ++ (2,0) node[linkw0] (2om8) {$2 \omega_8$} +(.4,.25) node[prime] {31}
  ++ (-1,-1) node[linkw0] (om7) {$\omega_7$} +(.38,.2) node[prime] {5}
  ++ (0,-1) node[wt] (om1) {$\omega_1$}
  ++ (0,-1) node[linkw0] (om8) {$\omega_8$} +(.38,.2) node[prime] {5}
  ++ (0,-1) node[linkw0] (om0) {$0$}
  ;
  %\end{pgfonlayer}

  % draw the lines
  \begin{pgfonlayer}{background}
    \pgfsetlinewidth{.8pt}
  \draw (om4) -- (om1om6)
  -- (2om1om8) -- (om3om8) -- (om2om7) -- (2om7)
  -- (om6om8) -- (om3om8) -- (om1om2)
  -- (om5) -- (om6om8) -- (om12om8)
  -- (3om8) -- (om7om8) -- (om2om8) -- (om3)
  -- (om6) -- (om1om8) -- (om2) -- (om7)
  -- (om1) -- (om8) -- (om0);

  \draw (om1om6) -- (om2om7);
  \draw (om3) -- (2om1) -- (om1om7) -- (om2om8);
  \draw (om5) -- (om1om7) -- (om12om8);
  \draw (om7om8) -- (om6);
  \draw (om1om8) -- (2om8) -- (om7);
  \end{pgfonlayer}
\end{tikzpicture}
\caption{Hasse diagram for $E_8$.} \label{figure:HasseE8}
\end{figure}
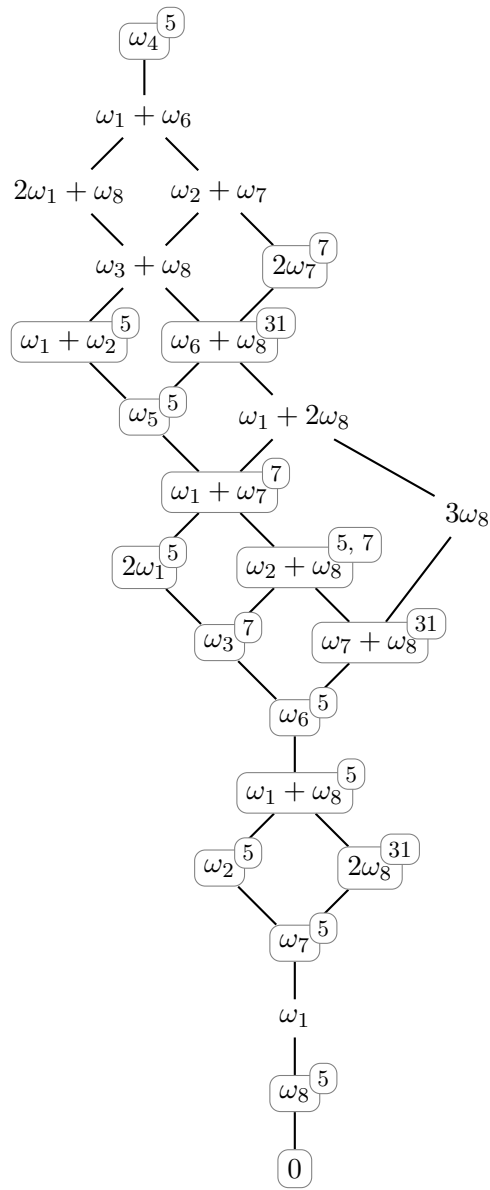

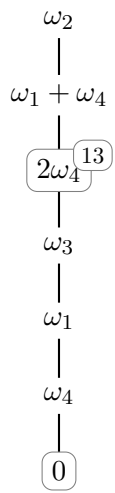
\begin{figure}[htbp]
\begin{tikzpicture}[scale=1]%
% \tikzstyle{wt}=[fill=red!20,text=black!85,rounded corners]
% \node[fill=white,scale=1.5] at (0,1) {$F_4$};
  \draw (0,0) node[wt] (om2) {$\omega_2$}
  -- ++ (0,-1) node[wt] (om1om4) {$\omega_1 + \omega_4$}
  -- ++ (0,-1) node[linkw0] (2om4){$2 \omega_4$} +(.45,.2) node[prime] {13} 
  + (0,0)
  -- ++ (0,-1) node[wt] (om3){$\omega_3$}
  -- ++ (0,-1) node[wt] (om1) {$\omega_1$}
  -- ++ (0,-1) node[wt] (om4) {$\omega_4$}
  -- ++ (0,-1) node[linkw0] (om0) {$0$}
  ;
\end{tikzpicture}
\caption{Hasse diagram for $F_4$} \label{figure:HasseF4}
\end{figure}

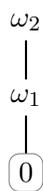
\begin{figure}[htbp]
\begin{tikzpicture}[scale=1]%
% \tikzstyle{wt}=[fill=yellow!80,text=black!85,rounded corners]
% \node[fill=white,scale=1.5] at (0,1) {$G_2$};
  \draw (0,0) node[wt] (om2) {$\omega_2$}
  -- ++ (0,-1) node[wt] (om1) {$\omega_1$}
  -- ++ (0,-1) node[linkw0](om0) {$0$}
  ;
\end{tikzpicture}
\caption{Hasse diagram for $G_2$} \label{figure:HasseG2}
\end{figure}

\clearpage

\subsection{Weights of induced modules} \label{subsection:fundamentalweights}

In this section we analyze the weights $\nu$ of induced modules $H^0(\tau)$, when the root system has rank $\leq 2$. Given a dominant weight $\tau$, and given a weight $\nu$ of $H^0(\tau)$, one has $\tau - \nu = \theta$ for some $\theta \in \N\Phi^+$. For each possible restriction $\tau$ of a weight $\lambda$ occurring in Section \ref{subsection:Hassediagrams} to a rank one or rank two root subsystem, we list the corresponding possible values for $\nu$ and $\theta$. When $\Phi$ has rank 2, write $\Delta = \set{\alpha_1,\alpha_2}$, and let $\omega_1$ and $\omega_2$ denote the corresponding fundamental dominant weights.

\bigskip

\noindent \textbf{Type $\mathbf{A_1}$.} In this case there is only a single fundamental dominant weight $\omega$. Write $\Delta = \set{\alpha}$. The weights of $H^0(\omega)$, $H^0(2\omega)$, and $H^0(3\omega)$ are given in Table \ref{table:A1}.

\begin{table}[htbp]
\begin{tabular}{||r|r||r|r||r|r||}
\hline
$H^{0}(\omega)$ &  & $H^{0}(2\omega)$   &   &  $H^{0}(3\omega)$ &    \\
\hline
\hline
$\nu$ & $\theta$ & $\nu$ & $\theta$  & $\nu$ & $\theta$ \\
\hline
$\omega$ & $0$ & $2\omega$ & $0$   &    $3\omega$ & $0$ \\
$-\omega$ & $\alpha$ &  $0$ & $\alpha$ &  $\omega$ & $\alpha$ \\
                    &                   & $-2\omega$ & $2\alpha$ &  $-\omega$ & $2\alpha$ \\
                    &                    &                      &                     &  $-3\omega$ & $3\alpha$ \\

\hline
\end{tabular}
\\[6pt]
\caption{Type $A_1$, weights of $H^0(\omega)$, $H^{0}(2\omega)$, $H^{0}(3\omega)$.} \label{table:A1}
\end{table}

\noindent \textbf{Type $\mathbf{A_1 \times A_1}$.} The weights of $H^0(\omega_1+\omega_2)$ and $H^0(2\omega_1+\omega_2)$ are given in Table \ref{table:A1A1}. The weights of $H^0(\omega_1)$, $H^0(2\omega_1)$, and $H^0(3 \omega_1)$ can be deduced from the data for Type $A_1$ in Table \ref{table:A1}. Since the situation is symmetric with respect to the ordering of the fundamental dominant weights, the weights of $H^0(\omega_2)$, $H^0(2\omega_2)$, and $H^0(3\omega_2)$ can also be deduced from the data in Table \ref{table:A1}.

\begin{table}[htbp]
\begin{tabular}{||r|r||r|r||r|r||r|r||}
\hline
$H^{0}(\omega_{1}+\omega_{2})$   &     & $H^{0}(2\omega_{1}+\omega_{2})$   &    \\
\hline 
$\nu$ & $\theta$ &$\nu$ & $\theta$     \\
\hline
$\omega_1+\omega_2$ & $0$   &$2\omega_1+\omega_2$ & $0$ \\
$\omega_1-\omega_2$ & $\alpha_2$ & $2\omega_1-\omega_2$ & $\alpha_2$ \\
$-\omega_1+\omega_2$ & $\alpha_1$ & $\omega_2$ & $\alpha_1$ \\
$-\omega_1-\omega_2$ & $\alpha_1+\alpha_2$ &  $-\omega_2$ & $\alpha_1+\alpha_2$ \\
&                                        &    $-2\omega_{1}+\omega_{2}$ & $2\alpha_{1}$ \\
&                                        &      $-2\omega_1-\omega_2$ & $2\alpha_1+\alpha_2$ \\                            
\hline
\end{tabular}
\\[6pt]
\caption{Type $A_1 \times A_1$, weights of $H^{0}(\omega_{1}+\omega_{2})$,  $H^{0}(2\omega_{1}+\omega_{2})$.} \label{table:A1A1}
\end{table}

\noindent \textbf{Type $\mathbf{A_2}$.} The weights of $H^0(\omega_1)$, $H^0(2\omega_1)$, and $H^0(3\omega_1)$ are given in Table \ref{table:A2}. Again, completely analogous results are obtained for the weights of $H^0(\omega_2)$, $H^0(2\omega_2)$, and $H^0(3\omega_2)$. The weights of $H^0(\omega_1+\omega_2)$ are also given in Table \ref{table:A2}.

\bigskip

\begin{table}[htbp]
\begin{tabular}{||r|r||r|r||r|r||r|r||}
\hline 
$H^{0}(\omega_{1})$ &     & $H^0(2\omega_1)$ &      & $H^0(3\omega_1)$   &         & $H^0(\omega_1+\omega_2)$    &         \\
\hline 
\hline
$\nu$ & $\theta$ & $\nu$ & $\theta$ & $\nu$ & $\theta$ & $\nu$ & $\theta$ \\
\hline
$\omega_1$ & $0$  & $2\omega_1$ & $0$ & $3\omega_1$ & 0 &$\omega_1+\omega_2$ & $0$ \\
$-\omega_1+\omega_2$ & $\alpha_1$ & $\omega_2$ & $\alpha_1$ & $\omega_1+\omega_2$ & $\alpha_1$ & $2\omega_1-\omega_2$ & $\alpha_2$ \\
$-\omega_2$ & $\alpha_1+\alpha_2$ & $\omega_1-\omega_2$ & $\alpha_1+\alpha_2$ & $2\omega_1-\omega_2$ & $\alpha_1+\alpha_2$ & $-\omega_1+2\omega_2$ & $\alpha_1$ \\
                         &                                         &            $-2\omega_1+2\omega_2$ & $2\alpha_1$ & $-\omega_1+2\omega_2$ & $2\alpha_1$ &  $0$ & $\alpha_1+\alpha_2$ \\
                         &                                         &      $-\omega_1$ & $2\alpha_1+\alpha_2$ & $0$ & $2\alpha_1+\alpha_2$ & $\omega_1-2\omega_2$ & $\alpha_1+2\alpha_2$ \\
                         &                                         &     $-2\omega_2$ & $2\alpha_1+2\alpha_2$ & $-3\omega_1+3\omega_2$ & $3\alpha_1$ & $-2\omega_1+\omega_2$ & $2\alpha_1+\alpha_2$ \\
                         &                                         &                                 &                                              & $\omega_1-2\omega_2$ & $2\alpha_1+2\alpha_2$ & $-\omega_1-\omega_2$ & $2\alpha_1+2\alpha_2$ \\
                         &                                         &                                 &                                              & $-2\omega_1+\omega_2$ & $3\alpha_1+\alpha_2$ &         &      \\
                         &                                         &                                 &                                              & $-\omega_1-\omega_2$ & $3\alpha_1+2\alpha_2$  &                                              &              \\
                         &                                         &                                 &                                              &   $-3\omega_2$ & $3\alpha_1+3\alpha_2$     &                  &     \\                                                                        
\hline
\end{tabular}
\\[6pt]
\caption{Type $A_2$, weights of $H^0(\omega_1)$,  $H^0(2\omega_1)$,  $H^0(3\omega_1)$,  $H^0(\omega_1+\omega_2)$.} \label{table:A2}
\end{table}

\noindent \textbf{Type $\mathbf{B_2}$.} Assume that $\alpha_1$ is long and that $\alpha_2$ is short. The weights of $H^0(\omega_1)$ and $H^0(\omega_2)$ are given in Table \ref{table:B2}.

\bigskip

\begin{table}[htbp]
\begin{tabular}{||r|r||r|r||}
\hline
$H^{0}(\omega_{1})$ &   & $H^{0}(\omega_{2})$ &     \\
\hline
\hline 
$\nu$ & $\theta$ &  $\nu$ & $\theta$ \\
\hline
$\omega_1$ & $0$ &  $\omega_2$ & $0$ \\
$-\omega_1+2\omega_2$ & $\alpha_1$ & $\omega_1-\omega_2$ & $\alpha_2$ \\
$0$ & $\alpha_1+\alpha_2$ & $-\omega_1+\omega_2$ & $\alpha_1+\alpha_2$ \\
$\omega_1-2\omega_2$ & $\alpha_1+2\alpha_2$ & $-\omega_2$ & $\alpha_1+2\alpha_2$ \\
$-\omega_1$ & $2\alpha_1+2\alpha_2$ &                &   \\
\hline
\end{tabular}
\\[6pt]
\caption{Type $B_2$, weights of $H^0(\omega_1)$, $H^0(\omega_2)$.} \label{table:B2}
\end{table}

\noindent \textbf{Type $\mathbf{G_2}$.} Assume that $\alpha_1$ is short and that $\alpha_2$ is long. The weights of $H^0(\omega_1)$ and $H^0(\omega_2)$ are given in Table~\ref{table:G2}.

\begin{table}[htbp]
\begin{tabular}{||r|r||r|r||}
\hline
$H^{0}(\omega_{1})$ &   & $H^{0}(\omega_{2})$  &     \\
\hline 
\hline
$\nu$ & $\theta$ &  $\nu$ & $\theta$ \\
\hline
$\omega_1$ & $0$ & $\omega_2$ & $0$ \\
$-\omega_1+\omega_2$ & $\alpha_1$ &  $3\omega_1-\omega_2$ & $\alpha_2$ \\
$2\omega_1-\omega_2$ & $\alpha_1+\alpha_2$ &  $\omega_1$ & $\alpha_1+\alpha_2$ \\
$0$ & $2\alpha_1+\alpha_2$ &  $-\omega_1+\omega_2$ & $2\alpha_1+\alpha_2$ \\
$-2\omega_1+\omega_2$ & $3\alpha_1+\alpha_2$ &  $2\omega_1-\omega_2$ & $2\alpha_1+2\alpha_2$ \\
$\omega_1-\omega_2$ & $3\alpha_1+2\alpha_2$ & $-3\omega_1+2\omega_2$ & $3\alpha_1+\alpha_2$ \\
$-\omega_1$ & $4\alpha_1+2\alpha_2$ & $0$ & $3\alpha_1+2\alpha_2$ \\  
                          &                                             &      $3\omega_1-2\omega_2$ & $3\alpha_1+3\alpha_2$ \\
                          &                                             &      $-2\omega_1+\omega_2$ & $4\alpha_1+2\alpha_2$ \\
                          &                                             &   $\omega_1-\omega_2$ & $4\alpha_1+3\alpha_2$ \\   
                          &                                             &    $-\omega_1$ & $5\alpha_1+3\alpha_2$ \\  
                          &                                             & $-3\omega_1+\omega_2$ & $6\alpha_1+3\alpha_2$ \\     
                          &                                             &  $-\omega_2$ & $6\alpha_1+4\alpha_2$ \\                                                        
\hline
\end{tabular}
\\[6pt]
\caption{Type $G_2$, weights of $H^0(\omega_1)$, $H^{0}(\omega_{2})$.} \label{table:G2}
\end{table}

\clearpage 

\section{VIGRE Algebra Group at the University of Georgia}

\subsection{} This project was initiated during Fall Semester 2009 under the Vertical Integration of Research and Education (VIGRE) Program sponsored by the National Science Foundation (NSF) at the Department of Mathematics at the University of Georgia (UGA). We would like to acknowledge the NSF VIGRE grant DMS-0738586 for its financial support of the project. The VIGRE Algebra Group at UGA consists of 5 faculty members, 3 postdoctoral fellows, and 7 graduate students. The group is led by Brian D.\ Boe, Jon F.\ Carlson, Leonard Chastkofsky, and Daniel K.\ Nakano. The email addresses of the group members are given below.

\bigskip
\begin{tabbing}
\hspace*{\parindent}\=Christopher M. Drupieski\ \ \ \= \kill
Faculty: \\[3pt]
\>Brian D.\ Boe \>brian@math.uga.edu \\
\>Jon F.\ Carlson \> jfc@math.uga.edu \\
\>Leonard Chastkofsky \> lenny@math.uga.edu \\
\>Daniel K.\ Nakano \> nakano@math.uga.edu \\
\>Lisa Townsley \> townsley@math.uga.edu\\[6pt]
Postdoctoral Fellows:  \\[3pt]
\>Christopher M.\ Drupieski \> cdrup@math.uga.edu \\
\>Niles Johnson \> njohnson@math.uga.edu \\
\>Benjamin F.\ Jones \> jonesbe@uwstout.edu \\[6pt]
Graduate Students:  \\[3pt]
\>Adrian M.\ Brunyate \> brunyate@math.uga.edu \\
\>Wenjing Li \> wli@math.uga.edu \\
\>Nham Vo Ngo \> nngo@math.uga.edu \\
\>Duc Duy Nguyen \> dnguyen@math.uga.edu \\
\>Brandon L.\ Samples \> bsamples@math.uga.edu \\
\>Andrew J.\ Talian \> atalian@math.uga.edu \\
\>Benjamin J.\ Wyser \> bwyser@math.uga.edu \\
\end{tabbing}

\subsection{Acknowledgements}

The authors would like to acknowledge useful discussions with Robert Guralnick on low degree cohomology for finite groups.

%\bibliographystyle{eprintamsmath}
%\bibliography{lowdegree}

\newcommand{\noopsort}[1]{}
\providecommand{\bysame}{\leavevmode\hbox to3em{\hrulefill}\thinspace}

\end{document}